\documentclass[reqno]{amsart}
\usepackage{amsmath,amsrefs,color}
\usepackage{geometry}
\usepackage{graphicx,picture,
	cleveref}
\usepackage[colorinlistoftodos]{todonotes}
\numberwithin{equation}{section}
\usepackage[title]{appendix}

\def\neweq#1{\begin{equation}\label{#1}}
\def\endeq{\end{equation}}

\newtheorem{theorem}{Theorem}[section]
\newtheorem{definition}{Definition}

\newtheorem{proposition}[theorem]{Proposition}
\newtheorem{lemma}[theorem]{Lemma}
\newtheorem{corollary}[theorem]{Corollary}

\newtheorem{remark}[theorem]{Remark}
\begin{document}

\title[Nonlinear elliptic equations in $\mathbb R^N\setminus \{0\}$ with Hardy potential ]{Sharp existence and classification results for nonlinear elliptic equations in $\mathbb R^N\setminus\{0\}$ with Hardy potential}

\author{Florica C. C\^irstea}

\address{Florica C. C\^irstea, School of Mathematics and Statistics, The University of Sydney, NSW 2006, Australia}
\email{florica.cirstea@sydney.edu.au}

\author{Maria F\u arc\u a\c seanu}

\address{Maria F\u arc\u a\c seanu, School of Mathematics and Statistics, The University of Sydney, NSW 2006, Australia}
\email{maria.farcaseanu@sydney.edu.au}

\date{\today}

\begin{abstract} 
For $N\geq 3$, by the seminal paper of Brezis and V\'eron (Arch. Rational Mech. Anal. 75(1):1--6, 1980/81), no positive solutions of 
$-\Delta u+u^q=0$ in $\mathbb R^N\setminus \{0\}$ exist if $q\geq N/(N-2)$; for $1<q<N/(N-2)$ 
the existence and profiles near zero of all positive $C^1(\mathbb R^N\setminus \{0\})$ solutions are given by Friedman and V\'eron (Arch. Rational Mech. Anal. 96(4):359--387, 1986). 

In this paper, 
for every $q>1$ and $\theta\in \mathbb R$, we
prove that the nonlinear elliptic problem $(\star)$ 
$-\Delta u-\lambda \,|x|^{-2}\,u+|x|^{\theta}u^q=0$ in $\mathbb R^N\setminus \{0\}$ with $u>0$ has a $C^1(\mathbb R^N\setminus \{0\})$ solution if and only if $\lambda>\lambda^*$, where $\lambda^*=\Theta(N-2-\Theta) $ with  $\Theta=(\theta+2)/(q-1)$.   
We show that (a) if $\lambda>(N-2)^2/4$, then  $U_0(x)=(\lambda-\lambda^*)^{1/(q-1)}|x|^{-\Theta}$ is the only solution of $(\star)$ and (b) 
if $\lambda^*<\lambda\leq (N-2)^2/4$, then all solutions of $(\star)$ are radially symmetric and their total set is   
$U_0\cup \{U_{\gamma,q,\lambda}:\ \gamma\in (0,\infty) \}$. We give the precise behavior of $ U_{\gamma,q,\lambda}$ 
near zero and at infinity, distinguishing between $1<q<q_{N,\theta}$ and $q>\max\{q_{N,\theta},1\}$, where $q_{N,\theta}=(N+2\theta+2)/(N-2)$. 

In addition, for $\theta\leq -2$ we settle the structure of the set of all 
positive solutions of $(\star)$ in $\Omega\setminus \{0\}$, subject to $u|_{\partial\Omega}=0$, where $\Omega$ is a smooth bounded domain containing zero, complementing the works of C\^{\i}rstea (Mem. Amer. Math. Soc.  227, 2014) and Wei--Du 
(J. Differential Equations 262(7):3864--3886, 2017).  

\end{abstract}
\maketitle

\section{Introduction and main results}\label{Sec1}

\subsection{Introduction} 
In two groundbreaking papers \cites{serrin64,serrin65}, Serrin studied {\em a priori} estimates of solutions, the removability of singularities, and the behavior of isolated singularities
for general quasilinear elliptic divergence-form equations
in $\Omega\setminus \{0\}$, where $\Omega\subseteq \mathbb R^N$ is a domain containing zero. 
The history of the isolated singularity problem, its challenges and significant achievements up to 1996 have been beautifully portrayed by V\'eron \cite{Veron1}. 
To address the interior isolated singularity problem for nonlinear elliptic equations under conditions outside the range of Serrin's papers is a very difficult and multifaceted task. This has fueled a lot of research in the last decade approaching the challenge from different viewpoints in specific and particular directions. 
A very active line of research (see, for example,  \cites{bidg,CRo,CRV,Cmem, CHEN2018,dan,fra1,fra2,Guerch,Invent,LD2019,wei-f})
is to explore the intricate links between the isolated singularity problem and 
singular potentials.  
Among these, the celebrated 
Hardy--Schr\"odinger operator (see $\mathbb L_\lambda$ in \eqref{line}) and, more generally, the Hardy--Sobolev operator play a prominent role.
Elliptic differential operators of this type are important in the famous Caffarelli--Kohn--Nirenberg inequalities,
being analyzed in 
connection with the best constants and symmetry (or symmetry breaking) of extremal functions \cites{cat,cat2}. 
Recent developments and challenges on such topics,  
which have significance to diverse areas such as 
quantum mechanics, astrophysics and Riemannian geometry,
are expounded by Ghoussoub and Robert in \cite{Frederic2}.

The Hardy-type inequalities among others reflect amazing mathematical structures in connection with a variety of ``energies" controlled by ``entropy" associated with the Laplacian. 
For a bounded domain $\Omega\subset \mathbb R^N$ $(N\geq 3)$ with $0\in \Omega$, the classical Hardy inequality states that 
\begin{equation} \label{har}  \int_\Omega |\nabla u|^2\,dx\geq \frac{(N-2)^2}{4}\,\int_\Omega \frac{u^2}{|x|^2}\,dx
\quad \mbox{for all } u\in H_0^1(\Omega).\end{equation}

It is well-known that $\lambda_H:=(N-2)^2/4$ is the best constant
for the inequality in \eqref{har}. However, $\lambda_H$ is never attained in $H_0^1(\Omega)$ when $\Omega$ is bounded, in which case a remainder was shown to exist by Brezis and V\'azquez \cite{bva}. 
The improvements of 
this inequality on bounded domains, involving for example the first zero of the Bessel function, have been lately linked with Sturm's theory regarding the oscillatory behavior of certain linear ordinary differential equations.
The Hardy inequality and its various refinements \cites{adi,ft,gh} have found rich and diverse applications including the analysis of the stability of solutions of semilinear elliptic and parabolic equations \cite{bva,cm}, 
the asymptotic behavior of the heat equation with singular potentials  \cite{vz} and
the stability of eigenvalues in elliptic problems with perturbed Schr\"odinger operators \cite{adi2}. For new perspectives and applications of functional inequalities, we refer to Ghoussoub and Moradifam's book \cite{Gh}.

In this paper,  
we study elliptic equations involving the Hardy--Schr\"odinger operator such as 
\begin{equation}\label{eq1}
-\Delta u-\frac{\lambda}{|x|^2}u+|x|^{\theta}u^q=0\quad \mbox{in }\Omega\setminus \{0\},\quad u>0\quad \mbox{in }\Omega\setminus\{0\},
\end{equation}
where $\Omega\subseteq \mathbb R^N$ ($N\geq 3$) is either $\mathbb R^N$ or an open set $\Omega_0$ that contains zero, or an open set $\Omega_\infty$ that contains $\{x\in \mathbb R^N:\ |x|>R\}$ for $R>0$; we   
assume $q,\lambda,\theta\in \mathbb R$ and focus on the super-linear case $q>1$. In the sub-linear case $0<q<1$, the classification and limit behaviors of the non-negative solutions of \eqref{eq1} in $\mathbb R^N\setminus \{0\}$ are known from the work \cite{bidg} of 
Bidaut-V\'eron and Grillot. 

From now on, we assume $q>1$. 
For the long history associated with the study of problem \eqref{eq1}, we refer to \cite{Cmem,Guerch,LD2017} and the references cited therein. 
When $\lambda=\theta=0$ and $q>1$, the study of the local and global solutions of \eqref{eq1} has been pioneered by Brezis--V\'eron \cite{brv} and V\'eron \cite{ver1,ver2}. Their results were generalized to $p$-Laplacian type equations with $1<p<N$ by 
Friedman and V\'eron \cite{frv}
for $p-1<q<N(p-1)/(N-p)$ and by V\'{a}zquez and V\'eron  \cite{vv} for $q\geq N(p-1)/(N-p)$. More recent generalizations exist in various directions, but without a Hardy potential   \cites{bra,cc,CC2015,CD2007,CD2010,song}.

As a major advance in this paper, we unveil 
the structure of the set of solutions of  \eqref{eq1} when $\Omega=\mathbb R^N$ for every $\theta,\lambda\in \mathbb R$ and $q>1$.
We give sharp existence results of all solutions of \eqref{eq1}, along with their precise behavior near the singular point $x=0$ and at infinity (see Theorems~\ref{uniq}, \ref{mult1} and \ref{non-e} or Corollary~\ref{ou1}). 
In addition, we provide the existence and profile near zero for 
all solutions of \eqref{eq1} when $\Omega$ is a bounded domain of $\mathbb R^N$ with $0\in \Omega$ and smooth boundary $\partial\Omega$ on which we impose a homogeneous (or non-homogeneous) Dirichlet boundary condition (see Theorem~\ref{est} or Corollaries~\ref{ouk3}--\ref{oa1}). Using the Kelvin transform, our results can be reformulated for problem \eqref{eq1} when $\Omega$ is an exterior domain. In Section~\ref{mr} we state our main results, which can be applied to equations where the Hardy--Schr\"odinger operator is replaced by more general operators, see Section~\ref{aplic}.  

The difficulties in deciphering the profiles near zero for the solutions of \eqref{eq1} with
$\Omega=\Omega_0$ arise 
from and vary according to the position of $\lambda$ with respect to $\lambda_H=(N-2)^2/4$, the best constant in the Hardy inequality, and the position of $\theta$ relative to $-2$.  
When $\theta>-2$, the asymptotic behavior near zero of the solutions of \eqref{eq1} has recently been 
classified by C\^{\i}rstea \cite{Cmem} for $\lambda\leq \lambda_H$ (relying on the fundamental solutions of the Hardy--Schr\"odinger operator), and by Wei and Du \cite{LD2017} for $\lambda>\lambda_H$. The methods in the latter case use among other things an approximation of $\lambda_H$ by first eigenvalues of suitably modified eigenvalue problems and 
cannot be applied to the former case and vice versa. 

We stress that for $\theta\leq -2$, unlike $\theta>-2$, every solution of \eqref{eq1} with $\Omega=\Omega_0$ is bounded near zero for every $\lambda\in \mathbb R$ (see Lemma~\ref{lem01} or \cite[Proposition~2.7]{LD2017}). Since singular solutions were the main interest of \cite{Cmem} and \cite{LD2017}, the precise asymptotic behavior near zero was not pursued there for $\theta\leq -2$ in \eqref{eq1}. What this means (when using the Kelvin transform) is that the asymptotic behavior {\em at infinity} for the solutions of \eqref{eq1} with $\Omega=\Omega_\infty$ is still open for $\theta>-2$.      
We settle this issue in Theorem~\ref{mar2} as a result of Theorem~\ref{mar}. 

By means of a new and unified approach, in  Theorem~\ref{mar} we recover and extend to every $\theta\leq -2$ the results in \cite{LD2017} for $\lambda>\lambda_H$ and also to the relevant maximal range for $\lambda\leq \lambda_H$ (see Section~\ref{mr} for details).

Our findings differ according to four cases: $(\mathcal U)$, $(\mathcal M_1)$, $(\mathcal M_2)$ and $(\mathcal N)$. The first one corresponds to $\lambda>\lambda_H$ and every $\theta\in \mathbb R$, whereas the latter three situations pertain to 
$\lambda\leq \lambda_H$ and arise from the position of $\theta$ with respect to two critical exponents
denoted by $\theta_-$ and $\theta_+$, where  
\begin{equation} \label{tat} \theta_\pm:=p_\pm\,(q-1)-2 \quad \mbox{and}\quad  
p_\pm:=\frac{N-2}{2}\pm\sqrt{\lambda_H-\lambda}.
\end{equation} 

A first difference can be remarked at this stage compared with previous studies: when $\lambda>\lambda_H$ in our approach we can deal with every $\theta\in \mathbb R$ at once. On the other hand, for $\lambda\leq \lambda_H$, we emphasize that the position of $\theta$ is not analyzed with respect to $-2$ but rather with two critical exponents $\theta_\pm$ defined in \eqref{tat}.    
This is because we rely on the Kelvin transform, the effect of which when applied to a solution of \eqref{eq1} is to render an equation of the same type as \eqref{eq1} in which only $\theta$ changes, becoming $\widehat{\theta}:=\left(N-2\right) q- (N+2+\theta)$. 
Remark that $p_++p_-=N-2$ and $p_+p_-=\lambda$. Thus, $p_\pm$ are the roots of $\ell=0$ seen as a quadratic equation in $\Theta$, where for every $\theta\in \mathbb R$, 
we define
\begin{equation} \label{sigma} \Theta:=\frac{\theta+2}{q-1}\quad \mbox{and}\quad \ell(\theta)=\ell:=\Theta^2 -\left(N-2\right)\Theta+\lambda.\end{equation}

In this paper, we give the structure of the set of {\em all} solutions of \eqref{eq1} with $\Omega=\mathbb R^N$ as follows: 
\begin{enumerate}
	\item[(a)] A uniqueness result in Theorem~\ref{uniq} for Case $(\mathcal U)$, where 
	\begin{enumerate}
		\item[$(\mathcal U)$] $\lambda>\lambda_H$ and $\theta\in \mathbb R$;
	\end{enumerate}
	
	\item[(b)] A multiplicity result in Theorem~\ref{mult1} in relation to Cases $(\mathcal M_1)$ and $(\mathcal M_2)$:
	\begin{enumerate}
		\item[$(\mathcal M_1)$] $\lambda\leq \lambda_H$ and $\theta<\theta_-$;
		\item[$(\mathcal M_2)$] $\lambda\leq \lambda_H$ and
		$\theta>\theta_+$. 
		(Here, we
		always have $\theta_+>-2$ since $q>1$.)
	\end{enumerate}
	\item[(c)] Non-existence of solutions of \eqref{eq1} in   Theorem~\ref{non-e} for Case $(\mathcal N)$, namely, 
	\begin{enumerate}
		\item[$(\mathcal N)$] $\lambda\leq \lambda_H$ and $\theta_-\leq \theta\leq \theta_+$. 
	\end{enumerate}
\end{enumerate} 

In Cases $(\mathcal U)$, $(\mathcal M_1)$ and $(\mathcal M_2)$, we see that $\ell>0$ and a 
radial solution of problem \eqref{eq1} is given by      
\begin{equation} \label{node}
U_0(x)=U_0(|x|):=\ell^{1/(q-1)} |x|^{-\Theta}\quad \mbox{for every  }x\in \mathbb R^N\setminus \{0\}.\end{equation}

The Kelvin transform in Section~\ref{mainr} reveals an intimate connection between Case $(\mathcal M_1)$ and Case~$(\mathcal M_2)$---the behavior near zero for a solution of \eqref{eq1} in one of these cases leads to knowledge of the behavior at infinity for another solution of \eqref{eq1} in the other case. In Case $(\mathcal M_2)$ the classification of the behavior near zero of the solutions of problem \eqref{eq1} given in \cite{Cmem} is closely linked with the fundamental solutions $\Phi_\lambda^\pm$ of 
the Hardy--Schr\"odinger operator $\mathbb L_\lambda:=\Delta+\lambda \,|x|^{-2}$. (We recall this classification in Theorem~\ref{th-cd}.)  
When $\lambda\leq \lambda_H$, let $\Phi_\lambda^\pm$ be the fundamental solutions of
the linear equation 
\begin{equation}\label{line} -\mathbb L_\lambda \Phi=0\quad \mbox{in } B_1(0)\setminus \{0\},\quad\mbox{where}\quad \mathbb L_\lambda:=\Delta+\lambda \,|x|^{-2}.\end{equation}
We set $\Phi_\lambda^-(x)=|x|^{-p_-}$ for each $x\in \mathbb R^N\setminus \{0\}$ and we define 
$\Phi_\lambda^+$ as follows
\begin{equation}\label{fundam}
\Phi_\lambda^+(x)=\left\{\begin{aligned}
& |x|^{-p_+}\ \mbox{ for every } x\in \mathbb R^N\setminus \{0\}&&\mbox{ if }
\lambda<\lambda_H,&\\
&|x|^{-\frac{N-2}{2}}\log \left(1/|x|\right)\ \mbox{ for every } 0<|x|<1&& 
\mbox{ if }
\lambda=\lambda_H.&
\end{aligned}
\right.
\end{equation} 
Then, $\Phi_\lambda^-$ satisfies \eqref{line} in $\mathcal D'(\mathbb R^N)$ and $\lambda |\cdot|^{-2}\Phi_\lambda^-(\cdot)$ is locally integrable in $\mathbb R^N$. (The same applies to $\Phi^{+}_\lambda$ if  $0<\lambda<\lambda_H$.)

\vspace{0.2cm}
For $\Omega=\mathbb R^N$, in Case $(\mathcal M_1)$, like in Case $(\mathcal M_2)$, we prove in Theorem~\ref{mult1} that \eqref{eq1} has infinitely many solutions, all radially symmetric satisfying 
$$ 
\lim_{|x|\to 0} \frac{u(x)}{U_0(x)}=1\ \ \mbox{in Case }(\mathcal M_1)\ \ \mbox{and}\ \ 
\lim_{|x|\to \infty} \frac{u(x)}{U_0(x)}=1 
\  \mbox{in Case }(\mathcal M_2).
$$

In addition, we find the following: 

\vspace{0.2cm}
$\bullet$ For every $\theta$ 
in Case $(\mathcal M_2)$, 
there exists a unique solution $u_{1,\theta}$ of \eqref{eq1} with $\Omega=\mathbb R^N$, subject to 
$ \lim_{|x|\to 0} u(x)/\Phi_\lambda^+(x)=1 
$, where $\Phi_\lambda^+$ is given by \eqref{fundam}.   
Then, $U_0\cup\{
\mu^\Theta u_{1,\theta}(\mu \cdot):
\ \mu\in (0,\infty)\}$ give all  
solutions of \eqref{eq1} in $\mathbb R^N\setminus \{0\}$. 

$\bullet$ For each $\theta$ in Case $(\mathcal M_1)$, the set of all solutions of \eqref{eq1} with $\Omega=\mathbb R^N$ is 
$U_0\cup \{ \mu^\Theta \,U_{1,\theta}(\mu\cdot):\ \mu\in (0,\infty)\}$, where $U_{1,\theta}$ is the Kelvin transform of $u_{1,\widehat{\theta}}$ with   $\widehat{\theta}:=(N-2)\,q-(N+2+\theta)$, namely,
$ U_{1,\theta}(x):=|x|^{2-N} \,u_{1,\widehat{\theta}}(x/|x|^2)$ 
for $x\in \mathbb R^N\setminus \{0\}.$
If $\theta$ is in Case $(\mathcal M_1)$, then $\widehat{\theta}$ is in Case $(\mathcal M_2)$ and $\ell(\theta)=\ell(\widehat{\theta})$. 

\vspace{0.2cm}
We bring to light 
several interesting features for the solutions of \eqref{eq1} when considered globally for  $\Omega=\mathbb R^N$ rather than locally, say for $\Omega=B_1(0)$. A significant difference between local and global solutions is that, whenever it exists, a solution of \eqref{eq1} in $\mathbb R^N\setminus \{0\}$ becomes radially symmetric (see Theorems~\ref{uniq} and \ref{mult1}). Another difference is that 
a host of solutions of \eqref{eq1} that exist in $\Omega\setminus \{0\}$, where $\Omega$ is a smooth bounded domain containing zero, cannot be extended as solutions of \eqref{eq1} in $\mathbb R^N\setminus \{0\}$. 
This becomes apparent by comparing 
Theorems~\ref{uniq}, \ref{mult1} and \ref{non-e} with Theorem~\ref{est} (see also Remark~\ref{mis}). 

\begin{definition}
	By a {\em solution} ({\em sub-solution}, {\em super-solution}) of \eqref{eq1}, we mean a positive function $u\in C^1(\Omega\setminus \{0\})$ such that  
	for all functions (non-negative functions) $\varphi\in C^1_c(\Omega\setminus \{0\})$, we have
	\begin{equation} 
	\label{dir}
	\int_\Omega \nabla u\cdot \nabla \varphi\,dx-\int_\Omega \frac{\lambda}{|x|^2}u\,\varphi\,dx+\int_\Omega
	|x|^{\theta}u^q\,\varphi\,dx=0\quad (\leq 0,\ \geq 0).
	\end{equation} 
\end{definition}
We denote by $C^1_c(\Omega\setminus \{0\})$ the set of functions in $C^1(\Omega\setminus \{0\})$ with compact support in $\Omega\setminus \{0\}$.

\subsection{Main results}\label{mr}

\vspace{0.2cm}
We show that Case $(\mathcal U)$ resembles Case $(\mathcal M_1)$
(respectively, Case~$(\mathcal M_2)$) only when it comes to the asymptotic behavior near zero (respectively, at infinity) for every solution of \eqref{eq1} with $\Omega=\Omega_0$ (respectively, $\Omega=\Omega_\infty$). In other respects, Case $(\mathcal U)$ is different from the rest. Our first result shows the reason.

\begin{theorem}[Uniqueness] \label{uniq} 
	In Case $(\mathcal U)$,     
	$U_0$ in \eqref{node} is the unique solution of \eqref{eq1} for $\Omega=\mathbb R^N$.
\end{theorem}

To our best knowledge, this result is completely new. The crux of the proof is to show that in Case $(\mathcal U)$, $U_0$ 
models  
the behavior near zero for every solution of \eqref{eq1}, which also happens in Case $(\mathcal M_1)$. 

\begin{theorem}[Classification of behavior near zero, Cases ${(\mathcal U})$ and  $(\mathcal M_1)$] \label{mar} 
	In Case $({\mathcal U})$ and 
	Case~$(\mathcal M_1)$, every solution of \eqref{eq1} with $\Omega=\Omega_0$ exhibits near zero the limit behavior 
	\begin{equation} \label{kop}	\lim_{|x|\to 0}\frac{u(x)}{U_0(x)}=1.\end{equation}   
\end{theorem}

The advance made in Theorem~\ref{mar} is to  
remove the restriction $\theta>-2$ imposed in \cite{Cmem} for Case~$(\mathcal M_1)$  and in \cite{LD2017} for Case $(\mathcal U)$. 
When $\theta\leq -2$ in 
Case~$(\mathcal U)$ a precise asymptotic behavior near zero remained open: it was shown only 
that every solution of \eqref{eq1} is bounded near zero \cite[Proposition 2.7]{LD2017} (see Lemma~\ref{lem01} for another proof). The method in \cite{LD2017} is different than in  \cite{Cmem} and neither treatment could be adapted to cover the full range of Theorem~\ref{mar}. Indeed, the techniques in \cite{Cmem} for Case~$(\mathcal M_1)$ rely on the fundamental solutions of the Hardy--Schrodinger operator $\mathbb L_\lambda$ in \eqref{line} and thus cannot be extended to tackle Case ~$(\mathcal U)$ in Theorem~\ref{mar}. On the other hand, the results in \cite{LD2017} depend on $\theta>-2$ and $\lambda>\lambda_H$ to get an approximation (performed in \cite{wei-f}) of the Hardy constant $\lambda_H$ by first eigenvalues of suitably modified eigenvalue problems of $-\Delta \phi=\lambda\,|x|^{-2}\phi$ in $H_0^1(\Omega)$, see  \cite[Lemma~2.3]{LD2017}; the analysis is also based on \cite[Proposition~2.5]{LD2017} giving the existence and uniqueness of the solution to \eqref{eq1}, where $\Omega$ is a smooth bounded domain containing zero, subject to a zero Dirichlet condition on $\partial\Omega$. But such a solution fails to exist in Case $(\mathcal M_1)$ as we shall prove in Theorem~\ref{est}. 

We provide a new and unified proof of Theorem~\ref{mar} that in our opinion is simpler than in \cite{Cmem} and \cite{LD2017}. We describe here the novelty of our approach. Unlike the method in \cite{Cmem}, we do not use the fundamental solutions of the operator $\mathbb L_\lambda$ in \eqref{line}. 
Instead, we construct explicit local sub-solutions and super-solutions of \eqref{eq1} with the advantage of
unifying the treatment of Cases~$(\mathcal M_1)$ and $(\mathcal U)$. We also reason differently than in  
\cite{LD2017} for $\theta>-2$ in Case~$(\mathcal U)$
since our proof of \eqref{kop} does not rely on the existence and uniqueness of the solution to 
\eqref{boun} with $h=0$.
In fact, we use the opposite strategy. First, without any concern for the existence issue, we prove that any solution of \eqref{eq1} satisfies \eqref{kop} in Cases $(\mathcal U)$ and $(\mathcal M_1)$. Then, with this precise behavior near zero, we infer in Theorem~\ref{est} that \eqref{boun} has at most one solution via the comparison principle in Lemma~\ref{compa}, whereas we obtain a solution as a limit of solutions to approximate boundary value problems. 
(For the existence of a solution in Case $(\mathcal M_1)$, the condition $h\not\equiv 0$ in \eqref{boun} is necessary.) 

What sets apart Case $(\mathcal U)$ and Case $(\mathcal M_1)$ from the remaining cases is that every solution $u$ of \eqref{eq1} with $\Omega=\Omega_0$ satisfies
\begin{equation} \label{tav} \liminf_{|x|\to 0} \frac{u(x)}{U_0(x)}>0.\end{equation} 
We prove this fact in Lemma~\ref{con1} by 
devising an explicit family   $\{w_\delta\}_{\delta>0}$ of ``rough" sub-solutions of \eqref{eq1} in
$\mathbb R^N\setminus \overline{B_\delta(0)}$,  satisfying $w_\delta=0$ on $\partial B_\delta(0)$ and $\lim_{\delta\to 0^+} w_\delta(x)=c\,U_0(x)$ for every $x\in \mathbb R^N\setminus \{0\}$, where $c>0$ is any suitably small constant. More precisely, we define $w_\delta$ as follows 
\begin{equation} \label{mic} w_\delta(x):=c\,U_0(x)\,\left[1-\left(\frac{\delta}{|x|}\right)^\alpha\right]^{\frac{1}{\sqrt{\alpha}}}\quad \mbox{ for every } |x|\geq \delta,\end{equation} where we fix $\alpha>0$ small,  depending only on $N,q,\theta$ and $\lambda$. It turns out that $w_\delta$ satisfies the above properties for every $c\in (0,c_\alpha)$, where $c_\alpha>0$ depends on $\alpha$, but not on $\delta$. With a suitable choice of the constant $c=c(r_0,u)$ such that $u\geq w_\delta$ on $\partial B_{r_0}(0)$, where $r_0>0$ is such that $\overline{B_{r_0}(0)}\subset \Omega$, the comparison principle in Lemma~\ref{compa} implies that $u\geq w_\delta$ for every $\delta\leq |x|\leq r_0$. By letting $\delta\to 0^+$, we obtain \eqref{tav}. We use the term ``rough" in relation to the sub-solution $w_\delta$ to indicate that at this stage we get $\liminf_{|x|\to 0} u(x)/U_0(x)\geq c$ for a constant $c>0$ that is {\em not} optimal. 

To complete the proof of Theorem~\ref{mar} it remains to show that \eqref{tav} yields \eqref{kop}. 
This is achieved in Proposition~\ref{lim-o} by a unified construction of refined local sub/super-solutions of \eqref{eq1} in $B_1(0)\setminus \{0\}$. These   
are explicitly given in \eqref{subso}, working in Cases $(\mathcal U)$, $(\mathcal M_1)$ and $(\mathcal M_2)$ (see Lemma~\ref{sun-s} for details). For every $\varepsilon\in (0,1)$ and $\eta>0$, we define $w_{\varepsilon,\eta}^\pm$ in $ \mathbb R^N \setminus \{0\}$ as follows 
\begin{equation} \label{subso} 
\begin{aligned}
& 
w_{\varepsilon,\eta}^-(x):= \left(1- \varepsilon\right) 
U_0(x)\,|x|^{\eta} \left( 1+\frac{|x|^\alpha}{\nu}\right)^{ -\frac{1}{\sqrt{\alpha}}},\\
& w_{\varepsilon,\eta}^+(x):= \left(1+ \varepsilon\right) 
U_0(x)\,|x|^{-\eta} \left( 1+\frac{|x|^\alpha}{\nu}\right)^{ \frac{1}{\sqrt{\alpha}}},
\end{aligned} 
\end{equation}
where  
$\alpha>0$ is suitably fixed, depending only on $N,q,\theta$ and $\lambda$, whereas $\nu>0$ is arbitrary. 
Such a construction, which we motivate in Section~\ref{motiv},  appears here for the first time and is robust enough to
deal with the Hardy potential in the 
nonlinear elliptic equation \eqref{eq1}. 

Theorem~\ref{uniq} follows readily from Theorem~\ref{mar}. Indeed, by proving in Case~$(\mathcal U)$ that there is only one asymptotic behavior near zero as in \eqref{kop}, using the Kelvin transform, we gain a unique behavior at infinity for {\em every} solution of \eqref{eq1} with $\Omega=\Omega_\infty$, namely, 
\begin{equation} \label{kop2} \lim_{|x|\to \infty} \frac{u(x)}{U_0(x)}=1. 
\end{equation} 
Then, using  Lemma~\ref{compa}, we derive that $U_0$ is in Case $(\mathcal U)$ the only solution of \eqref{eq1} in $\mathbb R^N\setminus \{0\}$.  
In contrast, we prove in Theorem~\ref{mult1} that 
\eqref{eq1} in $\mathbb R^N\setminus \{0\}$ has infinitely many solutions in  
Case~$(\mathcal M_1)$ and Case $(\mathcal M_2)$, whereas no solutions exist in Case $(\mathcal N)$ (see Theorem~\ref{non-e}). 

\vspace{0.2cm}
In Theorem~\ref{est} we extend some results from \cite{Cmem} and \cite{LD2017} and find new ones about the solutions of \eqref{eq1}, subject to a Dirichlet boundary condition on $\partial\Omega$. 
For every $q>1$ and $\lambda,\theta\in \mathbb R$, 
we address the existence, uniqueness or multiplicity of solutions to the nonlinear elliptic problem
\begin{equation} \label{boun} 
\left\{\begin{aligned}
& -\Delta u-\frac{\lambda}{|x|^2}u+|x|^{\theta}u^q=0\ \  \mbox{in }\Omega\setminus \{0\},&\\
& u=h\geq 0\ \  \mbox{on } \partial\Omega,\quad \ u>0 \ \ \mbox{in } \Omega\setminus \{0\},&
\end{aligned}\right.
\end{equation} 
where $\Omega$ is a smooth bounded domain containing the origin of $\mathbb R^N$ $(N\geq 3)$ and $h\in C(\partial\Omega)$ is a non-negative function. By a {\em solution} of \eqref{boun}, we mean a function
$u_h\in C^1(\Omega\setminus\{0\})\cap C(\overline{\Omega}\setminus \{0\})$ that is positive in $\Omega\setminus \{0\}$ such that $u_h|_{\partial\Omega}=h$ and $u_h$ satisfies \eqref{eq1} in
$\mathcal D'(\Omega\setminus\{0\})$. 

\begin{theorem}[Existence, uniqueness/multiplicity results for \eqref{boun}] \label{est} Let $\Omega\subseteq \mathbb R^N$ be a smooth bounded domain containing zero. Let $h\in C(\partial\Omega)$ be any non-negative function. 
	\begin{enumerate}
		\item[{\rm (1)}] Let Case $(\mathcal U)$ hold. Then, problem \eqref{boun} has a unique solution $u_h$. Moreover, if 
		$\Theta<(N-2)/2$ and $h\equiv 0$, then 
		$u_h(x)/|x|$ and $|x|^{\theta+1} u_h^q$ belong to $ L^2(\Omega)$, $u_h\in H^1_0(\Omega)$ and, for every $\varphi\in H_0^1(\Omega)$,   
		\begin{equation} 
		\label{ditt}
		\int_\Omega \nabla u_h\cdot \nabla \varphi\,dx-\int_\Omega \frac{\lambda}{|x|^2}u_h\,\varphi\,dx+\int_\Omega
		|x|^{\theta}u_h^q\,\varphi\,dx=0.
		\end{equation}  
		
		\item[{\rm (2)}] 
		Assume Case $(\mathcal M_1)$ or Case $(\mathcal N)$. If $h\not\equiv 0$ on $\partial \Omega$, then problem \eqref{boun} has a unique solution $u_h$. 

		\item[{\rm (3)}] If $h\equiv 0 $ on $\partial\Omega$, then \eqref{boun} has no solutions in Case $(\mathcal M_1)$ and Case $(\mathcal N)$.   
		
		\item[{\rm (4)}] Assume Case $(\mathcal M_2)$. Then, \eqref{boun} has infinitely many solutions: for every $\gamma\in (0,\infty]$ (also for $\gamma=0$ when $h\not\equiv 0$ on $\partial\Omega$), problem \eqref{boun}, subject to 
		\begin{equation} \label{gio} \lim_{|x|\to 0} \frac{u(x)}{\Phi_\lambda^+(x)}=\gamma 
		\end{equation} has a unique solution $u_h^{(\gamma)}$. Moreover,  for $\gamma=\infty$, the solution $u_h^{(\gamma)}$ satisfies 
		$$ \lim_{|x|\to 0}\frac{u_h^{(\gamma)}(x)}{U_0(x)}= 1.$$
		\begin{enumerate} 
			\item[{\rm (a)}] If $h\not\equiv 0$ on $\partial\Omega$, then $\{u_h^{(\gamma)}:\ 0\leq \gamma\leq \infty\}$ is the set of all solutions of problem \eqref{boun} and for $\gamma=0$ we have $\lim_{|x|\to 0} |x|^{p_-}u_h^{(\gamma)}(x)\in (0,\infty)$.
			\item[{\rm (b)}]  
			If $h= 0$ on $\partial\Omega$, then 
			all solutions of \eqref{boun} are
			$\{u_h^{(\gamma)}:\ 0<\gamma\leq \infty\}$. 
		\end{enumerate}
	\end{enumerate}
\end{theorem}

We point out that under the hypotheses of Theorem~\ref{est}, the behavior near zero for the unique solution $u_h$ of \eqref{boun} is provided by Theorem~\ref{mar} in  Cases~$(\mathcal U)$ and $(\mathcal M_1)$ and by Theorem~\ref{th-cd} in Case $(\mathcal N)$. 
The assertions (2) and (4) in Theorem~\ref{est} extend corresponding results in \cite{Cmem} for $\Omega=B_1(0)$. 
The novelty in Theorem~\ref{est} is given by the conclusions in (1) and (3). 

In contrast to Case~$(\mathcal U)$, the problem \eqref{boun} with $u=0$ on $\partial\Omega$ has no solutions in Case~$(\mathcal M_1)$. This can be shown using the Hardy inequality (see Remark~\ref{arg}) or by another argument relying on Theorem~\ref{mar}. 
Indeed, suppose that $u$ is a solution of \eqref{boun} with $u=0$ on $\partial\Omega$. Then, $\lim_{|x|\to 0} |x|^{p_-} \,u(x)=0$ in Case~$(\mathcal M_1)$ since $\Theta<p_-$. Hence, for every $\varepsilon>0$, we have $u(x)\leq \varepsilon |x|^{-p_-}$ for $|x|>0$ close to zero and for every $x\in\partial\Omega$ so that $0<u(x)\leq \varepsilon |x|^{-p_-}$ for every $x\in \Omega\setminus\{0\}$ in view of  Lemma~\ref{compa}. By letting $\varepsilon\to 0$, we arrive at $u\equiv 0$ in $\Omega\setminus\{0\}$, which is a contradiction. This argument can be easily adapted in Case $(\mathcal N)$ to establish the non-existence of solutions to \eqref{boun} with $h=0$.  


For $\theta>-2$ in Case $(\mathcal U)$ and $h=0$, the existence and uniqueness claim in Theorem~\ref{est} was proved differently by Wei and Du \cite[Proposition~2.5]{LD2017}. Their analysis relied 
on rough estimates \cite[Lemma 2.4]{LD2017}: there exist positive constants $C_1,C_2,r_0$ such that every solution $u$ of \eqref{eq1} satisfies
\begin{equation} \label{blo} C_1 |x|^{-\Theta}\leq u(x)\leq C_2 |x|^{-\Theta}\quad \mbox{for all } 0<|x|<r_0.\end{equation}
Then, arguing by contradiction, any solution of \eqref{boun} with $h=0$ was shown to coincide with its minimal solution $w$ via the strong maximum principle and a convexity trick of Marcus and V\'eron \cites{mv1,mv2}. 
The condition $\lambda>\lambda_H$ was essential in gaining the minimal solution $w$ as the limit $\delta\to 0^+$ of the unique solution $u_\delta$ to \eqref{eq1} in $\Omega^\delta:=\Omega\setminus \overline{B_\delta(0)}$, subject to $u=0$ on $\partial \Omega^\delta$.  
It was shown that $\liminf_{|x|\to 0} w(x)/U_0(x)>0$ by a comparison with the unique solution of \eqref{eq1} in a suitable annular domain with zero Dirichlet boundary condition. 
The first inequality of \eqref{blo} shows that 
every solution of \eqref{eq1} blows-up at zero by the assumption $\theta>-2$ in Case $(\mathcal U)$. The second inequality in \eqref{blo} was derived for the maximal solution $U$ of \eqref{eq1} satisfying $U=\infty$ on $\partial \Omega$, which was constructed in \cite{LD2017} as the limit ($\delta\to 0^+$) of the unique solution $U_\delta$ to \eqref{eq1} on the approximate domain $\Omega^\delta$ with boundary blow-up. 

In Theorem~\ref{est} we show that Case $(\mathcal U)$ is the maximal range for which \eqref{boun} with $h=0$ on $\partial \Omega$ has a unique solution.  
Hence, not only we give an alternative proof of  \cite[Proposition 2.5]{LD2017} for $\theta>-2$ in 
Case~$(\mathcal U)$ but also extend its existence and uniqueness conclusion to the entire Case $(\mathcal U)$. 

We now give some ideas behind our proof of Theorem~\ref{est}. 
As a byproduct of 
Theorem~\ref{mar} in Cases $(\mathcal U)$ and $(\mathcal M_1)$ (and of Theorem~\ref{th-cd} in Case $(\mathcal N)$), jointly with  Lemma~\ref{compa}, we find that \eqref{boun} has at most a solution.  We obtain a non-negative solution $u_h$ of \eqref{boun} as the limit when $k\to \infty$ of the unique positive solution $u_{h,k}$ of \eqref{eq1} in $\Omega\setminus \overline{B_{1/k}(0)}$, subject to $u=h$ on $\partial\Omega$ and $u=C\,|x|^{-\Theta}$ on $\partial B_{1/k}(0)$, where $C>0$ is a large constant.  
When $h\not\equiv 0$ on $\partial\Omega$ in Cases~$(\mathcal M_1)$ and $(\mathcal N)$, the positivity of $u_h$ in $\Omega$ follows from the strong maximum principle. When $h=0$ in Cases~$(\mathcal U)$ and $(\mathcal M_2)$, we prove that $u_h>0$ in $\Omega$ by showing that $u_{h,k}(x)\geq z_\delta(x)$ for every $1/k\leq |x|\leq \delta$ and every $k\geq k_0$ large enough, where $z_\delta$ is defined by 
\begin{equation} \label{zide} z_\delta(x):=c\,U_0(x)\left[1-
\left(\frac{|x|}{\delta}\right)^\alpha \right]^{\frac{1}{\sqrt{\alpha}}} \quad \mbox{for every } 0<|x|\leq \delta.
\end{equation}
(Here, like for $w_\delta$ in \eqref{mic}, we fix $\alpha>0$ small,  depending only on $N,q,\theta$ and $\lambda$.) 
By applying the Kelvin transform to 
the sub-solution $w_\delta$ 
of \eqref{eq1} in
$\mathbb R^N\setminus \overline{B_\delta(0)}$ for Cases $(\mathcal U)$ and $(\mathcal M_1)$, we obtain that $z_\delta$ is a sub-solution of \eqref{eq1} in $B_\delta(0)\setminus \{0\}$ for Cases $(\mathcal U)$ and $(\mathcal M_2)$.

In Case $(\mathcal M_2)$, the solution $u_h$  constructed above for problem \eqref{boun} is the maximal one since it satisfies $\liminf_{|x|\to 0} u_h(x)/U_0(x)>0$, which yields 
$\lim_{|x|\rightarrow 0}u_h(x)/U_0(x)=1$  via Proposition~\ref{lim-o}. 
Case~$(\mathcal M_2)$ is the only one when \eqref{boun} has infinitely many solutions (see Section~\ref{estp} for details).

\vspace{0.2cm}
We next return to \eqref{eq1} with $\Omega=\mathbb R^N$. Taking Case~$(\mathcal M_1)$ separately from 
Case~$(\mathcal M_2)$, 
we determine in Theorem~\ref{mult1} all solutions of \eqref{eq1} in $\mathbb R^N\setminus \{0\}$, together with their behavior near zero and at  infinity. If it were to exist, such a solution would satisfy at zero the limit behavior given by Theorem \ref{th-cd}, whereas at infinity the precise behavior  
listed in Theorem~\ref{mar2}.  

\begin{theorem}[Multiplicity, Cases $(\mathcal M_1)$ and $(\mathcal M_2)$] \label{mult1} Let $\Omega=\mathbb R^N$.  
	\begin{enumerate}
		\item[{\rm (1)}] Let Case $(\mathcal M_2)$ hold. For every $\gamma\in (0,\infty)$, equation \eqref{eq1}, subject to \eqref{gio}, has a unique solution $u_\gamma$. 
		All solutions of problem \eqref{eq1} satisfy \eqref{kop2} and are radially symmetric, being given by  
		$U_0$ and $\{u_\gamma:\ \gamma\in (0,\infty) \}$. 
		In addition, we have $ u_\gamma\leq u_{\gamma'}\leq U_0$ in  $\mathbb R^N\setminus \{0\}$ for every $0<\gamma<\gamma'<\infty$ and 
		$ U_0(x)=\lim_{\gamma\to \infty} u_{\gamma}(x)$ for each $ x\in \mathbb R^N\setminus \{0\}$. 
		\item[{\rm (2)}] Let Case $(\mathcal M_1)$ hold. For every $\gamma\in (0,\infty)$, equation \eqref{eq1}, subject to 
		\begin{equation} \label{e}
		\lim_{|x|\to \infty}   \frac{|x|^{N-2}\,u(x)}{\Phi_\lambda^+(1/|x|)}=\gamma
		\end{equation}
		has a unique solution, say $U_\gamma$. 
		All solutions of problem  \eqref{eq1} satisfy \eqref{kop} and are radially symmetric, being given by  
		$U_0$ and $\{U_\gamma:\ \gamma\in (0,\infty) \}$.  	 			 
	\end{enumerate} 			
\end{theorem}

By the Kelvin transform in \eqref{sum}, the claims of Theorem~\ref{mult1} in Case~$(\mathcal M_1)$ 
follow from those of Case~$(\mathcal M_2)$, the latter being treated in Proposition~\ref{prol}. 
Theorem~\ref{mult1} uncovers an unexpected feature: there are no solutions of \eqref{eq1} in $\mathbb R^N\setminus \{0\}$ satisfying $\lim_{|x|\to 0} |x|^{p_-}\,u(x)\in (0,\infty)$ for Case $(\mathcal M_2)$ and, hence, no solutions exist satisfying $\lim_{|x|\to \infty} |x|^{p_+}\,u(x)\in (0,\infty)$ for Case $(\mathcal M_1)$. 

We remark that in Case $(\mathcal M_2)$ (respectively, Case $(\mathcal M_1)$) of Theorem~\ref{mult1}, we can obtain the solutions $u_\gamma$ (respectively, $U_\gamma$) with $\gamma\in (0,\infty)$ from the solution corresponding to $\gamma=1$. We next make this point clear.  
Given $\mu>0$, let $T_\mu:C^1(\mathbb R^N\setminus \{0\})\to C^1(\mathbb R^N\setminus\{0\})$ be the operator defined by 
$$ T_\mu(u)(x):=\mu^\Theta u(\mu x)\quad \mbox{for every } x\in \mathbb R^N\setminus \{0\}. 
$$
Observe that whenever $U_0$ in \eqref{node} is well-defined such as in Cases $(\mathcal M_1)$ and $(\mathcal M_2)$, we have 
$ T_\mu(U_0)=U_0$. Moreover, the transformation $T_\mu$ sends a solution of \eqref{eq1} with $\Omega=\mathbb R^N$ into a solution of the same equation.  

Let $\Omega=\mathbb R^N$ and Case $(\mathcal M_2)$ hold. By Theorem~\ref{mult1}, there exists a unique solution $u_{1,\theta}$ of problem \eqref{eq1}, subject to $\lim_{|x|\to 0} u(x)/\Phi_\lambda^+(x)=1$. Moreover, $u_{1,\theta}$ is radially symmetric and satisfies $\lim_{|x|\to \infty} u_{1,\theta}(x)/U_0(x)=1$. Then, all solutions of \eqref{eq1} are given by $U_0$ and $\{T_\mu(u_{1,\theta}):\ 0<\mu<\infty\}$. 
(The solution $u_\gamma$ of \eqref{eq1}, subject to \eqref{gio}, corresponds to $T_\mu(u_{1,\theta})$ with $\mu=\gamma^{1/(\Theta-p_+)}$.) 

If Case $(\mathcal M_1)$ holds instead of Case $(\mathcal M_2)$, then all solutions of \eqref{eq1} are given by $U_0$ and $\{T_\mu(U_{1,\theta}):\ 0<\mu<\infty\}$, where $U_{1,\theta}$ is the Kelvin transform of $u_{1,\widehat{\theta}}$ with $\widehat{\theta}=(N-2)\,q-(N+2+\theta)$, namely, $U_{1,\theta}(x)=|x|^{2-N}\,u_{1,\widehat{\theta}}(x/|x|^2)$. 
(The solution $U_\gamma$ of \eqref{eq1}, subject to \eqref{e}, is $T_\mu(U_{1,\theta})$ with $\mu=\gamma^{1/(\Theta-p_-)}$.)     

\vspace{0.2cm} 
We next illustrate explicitly the findings of Theorem~\ref{mult1}.  

\vspace{0.2cm}
\noindent {\bf Example.} Fix $q>1$ and $-\infty<\lambda<\lambda_H$.

(i) In Case $(\mathcal M_2)$ if $\theta=\theta_++4\sqrt{\lambda_H-\lambda}$, then $ U_0\cup \{u_{\mu,\theta}:\ 0<  \mu<\infty\}$ represent all solutions of \eqref{eq1} in $\mathbb R^N\setminus \{0\}$, where we define 
$$ u_{\mu,\theta}(x):=
|x|^{-p_+}\left( \mu^{-2\sqrt{\lambda_H-\lambda}}+ [\ell(\theta)]^{-\frac{1}{2}}\,|x|^{2\sqrt{\lambda_H-\lambda}}\right)^{-\frac{2}{q-1}}\  \mbox{for }x\in \mathbb R^N\setminus \{0\}.
$$ 

(ii) In Case $(\mathcal M_1)$ if $\theta=\theta_--4\sqrt{\lambda_H-\lambda}$, then $U_0\cup \{U_{\mu,\theta}:\ 0< \mu<\infty\}$ is the set of all solutions of \eqref{eq1} in $\mathbb R^N\setminus \{0\}$, where we define 
$$ U_{\mu,\theta}(x):=
|x|^{-p_-}
\left( \mu^{2\sqrt{\lambda_H-\lambda}}+[\ell(\theta)]^{-\frac{1}{2}}\, |x|^{-2\sqrt{\lambda_H-\lambda}}\right)^{-\frac{2}{q-1}}\  \mbox{for } x\in \mathbb R^N\setminus \{0\}.
$$

In Case $(\mathcal N)$, we obtain 
that there are no solutions of \eqref{eq1} in $\mathbb R^N\setminus \{0\}$.  

\begin{theorem}[Non-existence, Case $(\mathcal N)$] \label{non-e}     
	Problem \eqref{eq1} with $\Omega=\mathbb R^N$ has no solutions in Case~$(\mathcal N)$.  	 
\end{theorem}

This non-existence result is somehow startling and it ensues essentially from $\Omega=\mathbb R^N$ in \eqref{eq1}. Theorem~\ref{est} shows that \eqref{eq1} in $B_1(0)\setminus \{0\}$ admits solutions exhibiting near zero each of the behaviors prescribed by Theorem~\ref{th-cd}. Yet, surprisingly, in Case $(\mathcal N)$ none of these local solutions can be extended as a solution of \eqref{eq1} in $\mathbb R^N\setminus \{0\}$. Were it to exist, a solution of \eqref{eq1} in $\mathbb R^N\setminus \{0\}$ would have the limit behavior near zero and at infinity given in Table~\ref{tab:table1}. 
Using essentially such precise asymptotics in Case~$(\mathcal N)$, we are able to rule out the existence of solutions of \eqref{eq1} for $\Omega=\mathbb R^N$. 


\begin{table}[h!]
	\caption{Possible profiles in Case $(\mathcal N)$}
	\label{tab:table1}
	\begin{tabular}{l|c|c} 
		\textbf{Case $(\mathcal N)$} & \textbf{behavior near zero in} & \textbf{behavior at infinity in}\\
		\hline
		$\theta_-<\theta<\theta+$ & \eqref{gin} & \eqref{gem}\\
		\hline
		$\theta=\theta_-<\theta_+$ & \eqref{110} & \eqref{gem}\\
		\hline
		$\theta=\theta_+>\theta_-$ & \eqref{gin} & \eqref{gem2}\\
		\hline
		$\theta=\theta_-=\theta_+$ & \eqref{111} & \eqref{gem3}\\
		\hline
	\end{tabular}
\end{table}


\subsection{Applications to weighted divergence-form equations} \label{aplic}

Here, we consider a related problem that can be solved using our method and results from Section~\ref{mr}. For $N\geq 3$, we study the nonlinear elliptic problem 
\begin{equation} \label{abx}
\left\{\begin{aligned}
& {\rm div}\, (|x|^{-2a}\nabla v)+d\,|x|^{-2\left(1+a\right)}\,v=|x|^{b}\,v^q && \mbox{in }\mathbb R^N\setminus \{0\} ,&\\
& v>0 &&   \mbox{in }\mathbb R^N\setminus \{0\},& 
\end{aligned}\right.
\end{equation}
where $a,b,d,q\in \mathbb R$, in the super-linear case $q>1$.

Before stating our main result on \eqref{abx}, we indicate what is known in the literature. 
For $b=d=0$ and $-1<a<(N-2)/2$, the influence of the weight $|x|^{-2a}$ in the divergence-form elliptic operator on the existence and local behavior near zero of the singular solutions of \eqref{abx} in $B_1(0)\setminus\{0\}$ follows from \cite{bra}: there exist positive solutions satisfying $\lim_{|x|\to 0} |x|^{N-2-2a}\,v_\gamma(x)=\gamma$ for some $\gamma\in (0,\infty]$ if and only if $1<q<N/(N-2-2a)$; in turn, if $q\geq N/(N-2-2a)$, then every solution of \eqref{abx} in $B_1(0)\setminus\{0\}$ can be extended as a positive continuous solution of \eqref{abx} in $B_1(0)$. In fact, more general weights than $|x|^{-2a}$ were considered in \cite{bra} using the framework of regular variation theory. 
For recent generalizations of these local existence and classification results to weighted quasilinear elliptic equations, see \cite{cc,song}. 

Returning to \eqref{abx} with $d=0$, we point out that the local behavior near zero has not been fully elucidated given that in the above-mentioned works, the parameters $a$ and $b$ are restricted to specific ranges (e.g., $a\leq (N-2)/2$ and $b>-N$), the focus being on the existence of singular solutions near zero (see, for example, Remark~1.1 in \cite{song}). Unfortunately, this limits our understanding of the behavior at infinity for the solutions of \eqref{abx}; if $v$ is a solution of \eqref{abx}, then by applying a generalized Kelvin transform, namely,  \begin{equation} \label{kge} \widehat{v}(x):=|x|^{2-N+2a}\,v(x/|x|^2), \end{equation}
it is readily seen that $\widehat{v}$ satisfies \eqref{abx} but with $b$ replaced by \begin{equation} \label{bhat}\widehat{b}:=(N-2-2a)\,q-(N+2a+b+2).\end{equation}

Using our main results regarding problem \eqref{eq1} with $\Omega=\mathbb R^N$, for every $a,b,d\in \mathbb R$ and $q>1$, we obtain in Theorem~\ref{extra0} a sharp criterion for the existence of solutions of \eqref{abx}, together with their exact profile near zero and at infinity. 
As a consequence, we derive that whenever they exist, {\em all} solutions of \eqref{abx} are radially symmetric. For ease of reference, we define
\begin{equation} \label{nota} \sigma:=\frac{2a+b+2}{q-1},
\quad \rho:=a-\frac{N-2}{2}\quad \mbox{and}\quad \ell:=(\sigma+\rho)^2-\rho^2+d. 
\end{equation}
We next state our main result concerning \eqref{abx}.

\begin{theorem} \label{extra0} 
	Problem \eqref{abx} has a solution if and only if $\ell>0$. 
	\begin{enumerate}
		\item[{\rm (i)}]
		If $d>\rho^2$, then problem \eqref{abx} has a unique solution given by 
		\begin{equation} \label{vze} v_0(x):=\ell^{\frac{1}{q-1}} |x|^{-\sigma}\quad \mbox{for every } x\in \mathbb R^N\setminus \{0\}. \end{equation}
		\item[{\rm (ii)}] If $(\sigma+\rho)^2>\rho^2-d\geq 0$ (for $\sigma\not=-\rho$), then \eqref{abx} has 
		infinitely many solutions, all radially symmetric and their total set is $v_0\cup \{v_\gamma:\ \gamma\in (0,\infty) \}$.
		For every $\gamma\in (0,\infty)$, we denote by $v_\gamma$ the unique solution of \eqref{abx} that satisfies the limit behavior near zero and at infinity given in Table~\ref{tab:table3}.
	\end{enumerate}
	\begin{table}[h!]
		\begin{center}
			\caption{Precise asymptotics for $v_\gamma$}
			\label{tab:table3}
			\begin{tabular}{c|l|c|c} 
				\textbf{Case} &
				\textbf{Criterion for existence} & \textbf{Behavior as $|x|\to 0$} & \textbf{Behavior as $|x|\to \infty$}\\
				
				\hline
				$(M_{11})$ &	 $d<\rho^2$,\  $\sigma+\rho<-\sqrt{\rho^2-d}$ & $\displaystyle \frac{v_\gamma(x)}{v_0(x)}\to 1$ & $\displaystyle\frac{v_\gamma(x)}{|x|^{\rho
						+\sqrt{\rho^2-d}}}\to \gamma$ \\			 
				\hline
				$(M_{21})$ &	$d<\rho^2$,\   $\sigma+\rho>\sqrt{\rho^2-d}$ & $ \displaystyle\frac{v_\gamma(x)}{|x|^{\rho-\sqrt{\rho^2-d}}}\to\gamma$ & $  \displaystyle\frac{v_\gamma(x)}{v_0(x)}\to 1$\\
				
				\hline
				$(M_{12})$	& $d=\rho^2$ and $\sigma<-\rho$ & $\displaystyle \frac{v_\gamma(x)}{v_0(x)}\to 1$ & $\displaystyle\frac{ v_\gamma(x)}{|x|^{\rho}\log |x|}\to \gamma$ \\	
				
				\hline
				$(M_{22})$ &	$d=\rho^2$ and $\sigma>-\rho$ & $\displaystyle \frac{\,v_\gamma(x)}{|x|^{\rho}\log (1/|x|)}\to \gamma$ & $\displaystyle \frac{v_\gamma(x)}{v_0(x)}\to 1$ \\	
				
				\hline
			\end{tabular}
		\end{center}
	\end{table}
\end{theorem}

\vspace{-0.2cm} 
Theorem~\ref{extra0} appears here for the first time except for $a=b=d=0$. 

We note the connection between various cases displayed in Table~\ref{tab:table3}. For a solution $v$ of \eqref{abx} in Case $(M_{11})$ (respectively, $(M_{12})$), its generalized Kelvin transform $\widehat{v}$ in \eqref{kge} is a solution of \eqref{abx} (with $b=\widehat{b}$ in \eqref{bhat}) in Case~$(M_{21})$ (respectively, $(M_{22})$). (Indeed, if we denote by $\widehat{\sigma}$ the value we obtain for $\sigma$ when $b$ is replaced by $\widehat{b}$, then $\widehat \sigma=-2\rho-\sigma$ and thus the condition $\sigma+\rho<-\sqrt{\rho^2-d}$ in $(M_{11})$ translates as $\widehat{\sigma}+\rho>\sqrt{\rho^2-d}$ in $(M_{21})$.) 

To obtain Theorem~\ref{extra0}, 
for a solution $v$ of \eqref{abx}, we use the transformation
\begin{equation} \label{tra} u(x):=|x|^{-a}\,v(x).
\end{equation}
Then, a direct calculation shows that $u$ is a positive solution of 
\begin{equation} \label{fip}  -\Delta u-\frac{\lambda}{|x|^2}u+|x|^{\theta}u^q=0\quad \mbox{in }\mathbb R^N\setminus \{0\},
\end{equation}
where $\lambda$ and $\theta$ are here given by 
$$ \lambda:=d+a\left(N-2-a\right)\quad \mbox{and}\quad \theta:=a\left(1+q\right)+b. 
$$
Hence, Theorem~\ref{extra0} follows by applying Theorems~\ref{uniq}, \ref{mult1} and \ref{non-e} for problem \eqref{fip}, then using the transformation in \eqref{tra}. 

\begin{remark}
	Due to \eqref{tra}, we can put problem \eqref{abx} in the same framework as in \eqref{eq1} and reformulate all our findings for \eqref{boun} in Theorem~\ref{est} to obtain corresponding conclusions for \eqref{abx} in $\Omega\setminus \{0\}$, subject to $u=h\geq 0$ on $\partial\Omega$, where $\Omega\subset \mathbb R^N$ is a smooth bounded domain containing zero and $h\in C(\partial\Omega)$. We leave these statements to the reader, who would then be able to get a full picture of all solutions of \eqref{abx} whether considered locally or globally.   
\end{remark}

In special cases, problems of the type  \eqref{abx} but with an opposite sign in the right-hand side of \eqref{abx} have been studied extensively by many authors motivated by    
applications to Riemannian geometry,    
as well as by various connections with the Caffarelli--Kohn--Nirenberg inequalities (e.g.,   \cite{dan,cat,cat2,Invent} and references therein); their treatment is based on variational or moving plane methods or uses the finite dimensional reduction of Lyapunov--Schmidt. 

In this paper, we follow a different approach since the sign in the right-hand side of \eqref{abx} does not allow us to use  
moving plane techniques, whereas variational methods cannot be employed here because of certain types of singularities that appear near zero for the solutions of \eqref{abx}. 

\vspace{0.2cm}
{\bf Structure of the paper.} In Theorem~\ref{th-cd} of Section~\ref{mainr}, we recall from \cite{Cmem} all the profiles near zero for the solutions of \eqref{eq1} with $\Omega=\Omega_0$. Based on this result and using   
the Kelvin transform and Theorem~\ref{mar}, we deduce in Theorem~\ref{mar2} the asymptotic behavior at infinity for the solutions of \eqref{eq1} with $\Omega=\Omega_\infty$.  
In Section~\ref{rou} we check that the functions $w_\delta$ and $z_\delta$ given in 
\eqref{mic} and \eqref{zide}, respectively are
sub-solutions of \eqref{eq1} on suitable domains. In Section~\ref{sect2} we include basic ingredients that will be often used in the sequel such as the comparison principle in Lemma~\ref{compa} and the {\em a priori} estimates in Lemma~\ref{lem01}.  
In Section~\ref{t31} we prove Theorem~\ref{uniq} and Theorem~\ref{mar}. In Section~\ref{estp} we establish the assertions of Theorem~\ref{est} on the existence of solutions 
of \eqref{boun}. We dedicate Section~\ref{multi} to the proof of Theorem~\ref{mult1}. The claim of  Theorem~\ref{non-e} is proved in Section~\ref{sec5}. We conclude the paper with comments and remarks in Section~\ref{com-rem}.

\section{Asymptotic behavior near zero / at infinity} \label{mainr}

For $\lambda\leq \lambda_H$ and $\theta>-2$
the sharp local behavior near zero and existence of solutions of \eqref{eq1} in $B_1(0)\setminus \{0\}$ is established in \cite{Cmem}, presenting a great diversity, which is recalled in Theorem~\ref{th-cd}. 
The study in \cite{Cmem} concerned more general nonlinear elliptic equations than \eqref{eq1} by invoking  regularly varying functions (the weight $|x|^\theta$ in \eqref{eq1} was replaced by a regularly varying function at zero with index $\theta>-2$). Some results in \cite{Cmem} such as those in Chapter 3.1 and the {\em a priori} estimates of Lemma~4.1 when applied to our equation \eqref{eq1} carry over beyond the range $\theta>-2$ (see Lemma~\ref{lem01} in Section~\ref{sect2}).


\begin{theorem}[See Chapter 7 in \cite{Cmem}] \label{th-cd}
	Let $\Omega=\Omega_0$, $\theta>-2$ and  
	$u$ be any solution of problem \eqref{eq1}.   
	\begin{enumerate}
		\item[{\rm (i)}] If Case $(\mathcal M_1)$ holds, then $u$ satisfies \eqref{kop};
		\item[{\rm (ii)}] 
		If Case $(\mathcal M_2)$ holds, then 
		exactly one of the following occurs:
		\begin{enumerate}
			\item[${\bf (A)}$] 
			$ 	\lim_{|x|\to 0}  |x|^{p_-}\,u(x)\in (0,\infty)$;
			\item[${\bf (B)}$] There exists $\gamma\in (0,\infty)$ such that $\lim_{|x|\to 0} u(x)/\Phi_\lambda^+(x)=\gamma$;
			\item[${\bf (C)}$] $u$ satisfies \eqref{kop}. 
		\end{enumerate} 
		\vspace{0.2cm}
		\item[{\rm (iii)}]
		Assume Case $(\mathcal N)$. 
		Then, $\lim_{|x|\to 0} u(x)/\Phi_\lambda^+(x)=0$ and we have
		\begin{enumerate}
			\item[$(\mathcal N_1)$] If  
			$\theta_-<\theta\leq \theta_+$ for $\lambda<\lambda_H$, then
			\begin{equation} \label{gin} \lim_{|x|\to 0} |x|^{p_-}\,u(x)\in (0,\infty);\end{equation}  
			\item[$(\mathcal N_{2})$] If 
			$\theta=\theta_-<\theta_+$ for
			$\lambda<\lambda_H$, then 
			$u$ satisfies
			\begin{equation} \label{110}
			\lim_{|x|\to 0} |x|^{p_-} \left (\log \frac{1}{|x|}\right )^{\frac{1}{q-1}}u(x)=\left (\frac{N-2-2p_-}{q-1}\right )^{\frac{1}{q-1}};
			\end{equation}
			\item[$(\mathcal N_{3})$] If $\theta=\theta_-=\theta_+$ for $\lambda=\lambda_H$, then $u$ satisfies
			\begin{equation} \label{111}
			\lim_{|x|\rightarrow 0} |x|^{\frac{N-2}{2}}\left (\log \frac{1}{|x|}\right )^{\frac{2}{q-1}}u(x)=\left [\frac{2\left(q+1\right)}{(q-1)^2}\right ]^{\frac{1}{q-1}}.
			\end{equation}
		\end{enumerate}
	\end{enumerate}
\end{theorem}

We remark that the condition $\theta>-2$ in Theorem~\ref{th-cd} can be removed (relevant for Case~$(\mathcal M_1)$ and items $(\mathcal N_1)$ and $(\mathcal N_2)$ in (iii)) and the conclusions extended according to the specified case. 
We indicate why in Case $(\mathcal M_1)$ the condition $\theta>-2$ is not needed to reach \eqref{kop}. The idea in \cite{Cmem} is to reduce the proof of \eqref{kop} to the case of radially symmetric solutions $u(r)=u(|x|)$ in $B_1(0)\setminus \{0\}$ and for these to use a suitable change of variable: 
$$ 
y(s)=u(r)/\Phi_\lambda^-(r)\mbox{ with }s= \Phi_\lambda^+(r)/\Phi_\lambda^-(r).
$$ 
In Case $(\mathcal M_1)$,  the {\em a priori} estimates in Lemma~\ref{lem01} (see Section~\ref{sect2}) imply that $\lim_{r\to 0^+} u(r)/\Phi_\lambda^+(r)=0$. Since     
$\lim_{\tau\to 0} \int_\tau^1 r^{1+\theta-(q-1)p_-}\,dr=\infty$ if $\lambda<\lambda_H$ and 
$\lim_{\tau\to 0}\int_\tau^{1/2} r^{\theta+[N-q(N-2)]/2}\,\log\,(1/r)\,dr=\infty $ if $\lambda=\lambda_H$,  then
Theorem 1.1 in Taliaferro \cite{tal} gives that any two positive solutions of the differential equation satisfied by $y$ are asymptotically equivalent at $\infty$. This means that every positive radial solution $u$ of \eqref{eq1} in $B_1(0)\setminus \{0\}$ satisfies $\lim_{r\to 0^+}u(r)/U_0(r)=1$. 
The ingredients used to reduce to the radial case work for every $\theta\in \mathbb R$, see Lemma~\ref{lem01}, Lemma \ref{regu} and Remark~\ref{r-regu}.

The claim of Theorem~\ref{th-cd} in Case $(\mathcal N_{1})$ holds for every $\theta_-<\theta\leq \theta_+$ with the same proof as for $\max\{-2,\theta_-\}<\theta\leq \theta_+$. Similarly, 
with the methods in \cite{Cmem}, the assertion of Case $(\mathcal N_2)$, which was proved for $\theta=\theta_-$ and $0<\lambda<\lambda_H$, remains valid for $\lambda\leq 0$.

\vspace{0.2cm}
From Theorem~\ref{mar} and Theorem \ref{th-cd}, we gain full understanding of the limit behavior near zero for all solutions of \eqref{eq1} for every $\theta,\lambda\in \mathbb R$ and $q>1$. 
This and the Kelvin transform allow us to classify the local behavior at infinity for every solution of \eqref{eq1} as follows.  

\begin{theorem} 
	[Classification of the behavior at $\infty$] \label{mar2} 
	Suppose that $u$ is an arbitrary solution of \eqref{eq1} with $\Omega=\Omega_\infty$.
	\begin{description}	
		\item[$\bullet$]	In Case $({\mathcal U})$ and Case $(\mathcal M_2)$, we have \eqref{kop2}.			 
		\item[$\bullet$] In Case $(\mathcal M_1)$, exactly one of the following behaviors occurs:
		\begin{enumerate}
			\item[${\bf (D)}$] $\lim_{|x|\to \infty} |x|^{p_+}\,u(x)\in (0,\infty)$;
			\item[${\bf (E)}$] There exists $\gamma\in (0,\infty)$ such that $\lim_{|x|\to \infty}   |x|^{N-2}\,u(x)/\Phi_\lambda^+(1/|x|)=\gamma$;
			\item[${\bf (F)}$] $u$ satisfies \eqref{kop2}.  
		\end{enumerate}		
		\item[$\bullet$] In Case $(\mathcal N)$, we distinguish three situations:
		\begin{enumerate}
			\item If $\theta_-\leq \theta<\theta_+$, then 
			\begin{equation} \label{gem} \lim_{|x|\to \infty} |x|^{p_+}u(x)\in (0,\infty);\end{equation}  
			\item If $\theta=\theta_+$ and $\lambda<\lambda_H$, then $u$ satisfies
			\begin{equation} \label{gem2} \lim_{|x|\to \infty} |x|^{p_+}\left(\log |x|\right)^{\frac{1}{q-1}} \,u(x)=\left(\frac{N-2-2p_-}{q-1}\right)^{\frac{1}{q-1}};
			\end{equation}
			\item If $\theta=\theta_+$ and $\lambda=\lambda_H$, then
			\begin{equation} \label{gem3}  
			\lim_{|x|\to \infty} |x|^{\frac{N-2}{2}}\left (\log |x|\right )^{\frac{2}{q-1}}u(x)=\left [\frac{2\left(q+1\right)}{(q-1)^2}\right ]^{\frac{1}{q-1}}.
			\end{equation}
		\end{enumerate} 		
	\end{description}		 
\end{theorem}	

The existence of all the profiles at infinity prescribed by Theorem~\ref{mar2} follows from Theorem~\ref{est} and the Kelvin transform. 

\vspace{0.2cm}
{\bf The Kelvin transform.} 
Let $\Omega=\mathbb R^N$ with $N\geq 3$. For a solution $u$ of \eqref{eq1}, let $u_*$
be its Kelvin transform with respect to the unit sphere in $\mathbb R^N$: 
\begin{equation} \label{sum} u_*(x):=|x|^{2-N}\,u(x/|x|^2)\quad \mbox{for }x\in \mathbb R^N\setminus \{0\}.  
\end{equation}
Since $\Delta u_*(x)=|x|^{-N-2} (\Delta u)\,(x/|x|^2)$, we obtain that $u_*$ satisfies an equation of the same type as $u$, where $\theta$ is replaced by $\widehat{\theta}:=\left(N-2\right) q-(N+2+\theta)$. In other words, we have 
\begin{equation} \label{ahu} 
-\Delta u_*-\frac{\lambda}{|x|^2}u_*+|x|^{\widehat{\theta}}\, u_*^q=0\quad \mbox{in }\mathbb R^N\setminus \{0\}.  
\end{equation}
The behavior of $u_*$ near zero (respectively, at infinity) is obtained from the behavior of $u$ at infinitely (respectively, near zero) by using \eqref{sum}. 
For such conversions, it is useful to keep in mind that 
\begin{equation} \label{lig} \widehat{\Theta}:=\frac{ \widehat{\theta}+2}{q-1}=N-2-\Theta
\quad\mbox{and}\quad \widehat{\ell}:=\widehat{\Theta}^2 -(N-2)\,\widehat{\Theta}+\lambda=\ell. 
\end{equation} 
We see that if $\widehat{\theta}\leq -2$, then $\theta>-2$ 
and similarly, $\theta\leq -2$ implies that $\widehat{\theta}>-2$ since $q>1$. 
In addition, if $\lambda\leq \lambda_H$, we have $\widehat{\theta}=\theta_++\theta_--\theta$ using $\theta_\pm$ in \eqref{tat}; thus, $\widehat{\theta}<\theta_-$ is equivalent to $\theta>\theta_+$, whereas $\widehat{\theta}>\theta_+$ if and only if $\theta<\theta_-$. 

What this means is that in Case $(\mathcal U)$ (Case $(\mathcal M_1)$ and Case $(\mathcal M_2)$, respectively)
if $u$ is a solution of \eqref{eq1}, then its behavior at zero leads (through its Kelvin transform $u_*$) to knowledge of the behavior at infinity for some other solution of \eqref{eq1} in Case $(\mathcal U)$ (Case $(\mathcal M_2)$ and Case $(\mathcal M_1)$, respectively).
Then, Theorem~\ref{mar} implies in 
Case $(\mathcal U)$ and Case $(\mathcal M_2)$ 
that every solution $u$ of \eqref{eq1} satisfies a unique behavior at infinity given by \eqref{kop2} using the Kelvin transform in \eqref{sum} and \eqref{lig}. 


\section{Construction of explicit ``rough" sub-solutions}
\label{rou}

In Lemma~\ref{con1} we proceed with the explicit construction and verification of the ``rough" sub-solutions $w_\delta$ for \eqref{eq1} on exterior domains in Cases $(\mathcal U)$ and $(\mathcal M_1)$. This will find application in the proof of Theorem~\ref{mar} via Corollary~\ref{fum}. In addition, through the Kelvin transform, we immediately acquire corresponding sub-solutions $z_\delta$ of \eqref{eq1} for $0<|x|\leq \delta$ in Cases $(\mathcal U)$ and $(\mathcal M_2)$, which will be used in the proof of Theorem~\ref{est}. 

\begin{lemma} \label{con1} 
	In Case $(\mathcal U)$ and Case $(\mathcal M_1)$, for every $\alpha>0$ small, depending on $N$, $q$, $\theta$ and $\lambda$, there exists $c_\alpha>0$ such that for every constant $c\in (0,c_\alpha)$ and all $\delta>0$, the function $w_\delta$ given by  
	\begin{equation} \label{wd} w_\delta(x):=c\,U_0(x)\left[1-\left(\frac{\delta}{|x|}\right)^\alpha \right]^{\frac{1}{\sqrt{\alpha}}} \quad \mbox{for every } |x|\geq \delta
	\end{equation} satisfies the following inequality
	\begin{equation}\label{da}
	-\mathbb L_\lambda(w_\delta)+|x|^\theta\,(w_\delta)^q
	\leq 0\quad\mbox{for every } |x|>\delta.
	\end{equation}
\end{lemma}

\begin{proof}
	Assume Case $(\mathcal U)$ or Case $(\mathcal M_1)$. 
	Recall that $\ell>0$. 
	For every $\alpha>0$, we use the notation
	\begin{equation} \label{abali}
	\begin{aligned}
	& A_\alpha:=1-\frac{\sqrt{\alpha}}{\ell}\left(N-2-2\,\Theta -\sqrt{\alpha}\right) \\ 
	& B_\alpha:=-2+\frac{\sqrt{\alpha}}{\ell}\left(N-2-2\,\Theta-\alpha\right).	
	\end{aligned}
	\end{equation}
	Using the definition of $w_\delta$ in \eqref{wd}, for every $|x|>\delta$, we obtain that 
	$$ \mathbb{L}_\lambda(w_\delta)=c\,\ell^\frac{q}{q-1}|x|^{-\Theta-2}\left[1-\left(\frac{\delta}{|x|}\right)^\alpha \right]^{\frac{1}{\sqrt{\alpha}}-2}\left[A_\alpha 
	\left(\frac{\delta}{|x|}\right)^{2\alpha}+B_\alpha \left(\frac{\delta}{|x|}\right)^\alpha+1\right]$$
	as well as
	$$ |x|^\theta (w_\delta(x))^q=c^q\,\ell^{\frac{q}{q-1}}|x|^{-\Theta-2}\left[1-
	\left(\frac{\delta}{|x|}\right)^\alpha \right]^{\frac{q}{\sqrt{\alpha}}}.
	$$
	For every $t\in (0,1)$, we define 
	$$ h_\alpha(t):=
	\left(1-t\right)^{-\frac{1}{\sqrt{\alpha}}(q-1)-2}\left(A_\alpha\, t^{2}+
	B_\alpha\, t+1\right). 
	$$
	If $t=(\delta/|x|)^\alpha$, then the inequality in \eqref{da} is equivalent to 
	\begin{equation}\label{retin}
	h_\alpha(t)\geq c^{q-1}\quad\text{for every }t\in (0,1).
	\end{equation}
	We prove below that we can choose $\alpha\in (0,1)$ small, depending only on $N,q,\theta$ and $\lambda$, such that 	
	\begin{equation}\label{mnb}
	\inf_{t\in (0,1)}h_\alpha(t)>0.
	\end{equation} 
	Let $\alpha\in (0,1)$ be small such that $A_\alpha>0$. 
	Observe that  
	\begin{equation} \label{fior}  A_\alpha+B_\alpha+1= \frac{\alpha}{\ell}\left(1-\sqrt{\alpha}\right)>0.  
	\end{equation}
	By a simple computation, using \eqref{abali}, we find that
	$$
	B_\alpha^2-4A_\alpha=\frac{4\,\alpha}{\ell^2}\left \{\lambda_H-\lambda+\ell\sqrt{\alpha} +\frac{\alpha}{4}\left[\alpha-2\left(N-2-2\,\Theta\right)\right]\right \}.
	$$
	To prove \eqref{mnb}, we analyze Case $(\mathcal U)$ separately from Case $(\mathcal M_1)$. 
	\begin{enumerate}
		\item[$(\mathcal U)$] Let $\lambda>\lambda_H$ and $\theta\in \mathbb R$. Then, we have $B_\alpha^2-4A_\alpha<0$ by choosing $\alpha>0$ small enough. Hence, $A_\alpha \,t^2+B_\alpha\, t+1>0$ for every $t\in \mathbb R$. 
		
		\item[$(\mathcal M_1)$] Let $\lambda\leq \lambda_H$ and $\theta<\theta_-$. Then for $\alpha>0$ small enough,  we have that $B_\alpha^2-4A_\alpha > 0$ and, hence, the quadratic equation $A_\alpha \,t^2+B_\alpha \,t+1=0$ has two distinct roots,  
		say $t_1(\alpha)$ and $t_2(\alpha)$. 
		Since $\theta<\theta_-$ is equivalent to  $\Theta<p_-$, from  $p_-\leq (N-2)/2$, we obtain that $N-2-2\,\Theta>0$. Hence, we have  
		$t_1(\alpha)\,t_2(\alpha)=1/A_\alpha>1$ for $\alpha\in (0,1)$ small enough. 
		Then, both roots $t_1(\alpha)$ and $t_2(\alpha)$ are greater than $1$ in view of \eqref{fior}. 
		Hence, we have 
		\begin{equation} \label{fig} A_\alpha\,t^2+B_\alpha \,t+1>0\quad \mbox{for every } t\in [0,1].\end{equation} 
	\end{enumerate}
	Since \eqref{fig} holds for Case $(\mathcal U)$ and Case $(\mathcal M_1)$, using that $(1-t)^{-\frac{1}{\sqrt{\alpha}}(q-1)-2}\geq 1$ for every $t\in [0,1)$, we deduce \eqref{mnb}.
	
	\vspace{0.2cm}
	We define $c_\alpha=\left(\inf_{t\in (0,1)}h_\alpha(t)\right)^{1/(q-1)}>0$. Then, for every $0<c<c_\alpha$, we obtain \eqref{retin}.  
	This ends the proof of Lemma~\ref{con1}. 
\end{proof}

\begin{corollary} \label{fum}
	Let $\Omega=\Omega_0$. 
	In Case $(\mathcal U)$ and Case $(\mathcal M_1)$,     
	every solution $u$ of problem \eqref{eq1} satisfies 
	\begin{equation} \label{limi} \liminf_{|x|\to 0} \frac{u(x)}{U_0(x)}>0. \end{equation} 
\end{corollary} 

\begin{proof}	
	Let $r_0\in (0,1)$ be such that 
	$\overline{B_{r_0}(0)}\subset \Omega$. 
	Choose $\alpha\in (0,1)$ as in Lemma~\ref{con1}, according to which there exists $c_\alpha>0$ such that for every $c\in (0,c_\alpha)$ and all $\delta>0$, the function $w_\delta$ in \eqref{wd} satisfies
	\eqref{da}. 
	Let $\delta\in (0,r_0)$ be arbitrary.  
	Choose $0<c<\min\{c_\alpha, \min_{x\in \partial B_{r_0}(0)} (u(x) /U_0(x))\}$. 
	Clearly, $ w_\delta=0<u$ on $ \partial B_\delta(0)$. 
	Our choice of $c$ gives that $ u\geq w_\delta$ on $\partial B_{r_0}(0)$.  
	We now apply the comparison principle in Lemma \ref{compa} to obtain that 
	\begin{equation}\label{fd}
	u(x) \geq 	w_\delta(x)\quad\mbox{for every } \delta\leq |x|\leq r_0. 
	\end{equation}
	For any $x\in B_{r_0}(0)\setminus \{0\}$, by letting $\delta\rightarrow 0$ in \eqref{fd}, we deduce that
	$$
	u(x)\geq c\,U_0(x)  \quad \mbox{for every } x\in B_{r_0}(0)\setminus \{0\},
	$$
	which finishes the proof of \eqref{limi}. \end{proof}

Our next result is important in the proof of Theorem~\ref{est} to treat Case~($\mathcal U$) with $h\equiv 0$ in \eqref{boun} (see Lemma~\ref{nonz}) and to analyze Case $(\mathcal M_2)$ in Lemma~\ref{m2c}.  

\begin{lemma} \label{con2} 
	In Case $(\mathcal U)$ and Case $(\mathcal M_2)$ for every $\alpha>0$ small, depending on $N$, $q$, $\theta$ and $\lambda$, there exists $c_\alpha\in (0,1)$ such that for any $c\in (0,c_\alpha)$ and all $\delta>0$, the function $z_\delta$ given by  
	\begin{equation} \label{zde} z_\delta(x):=c\,U_0(x)\left[1-
	\left(\frac{|x|}{\delta}\right)^\alpha \right]^{\frac{1}{\sqrt{\alpha}}} \quad \mbox{for every } 0<|x|\leq \delta
	\end{equation} satisfies the following inequality
	\begin{equation}\label{da2}
	-\mathbb L_\lambda(z_\delta)+|x|^\theta\,(z_\delta)^q
	\leq 0\quad\mbox{for every } 0<|x|<\delta.
	\end{equation}
\end{lemma}

\begin{proof} The claim follows from Lemma~\ref{con1} by using the Kelvin transform.
\end{proof}

\section{Basic ingredients} \label{sect2}

We often use the following comparison principle, which is a consequence of Lemma 2.1 in \cite{CR2004}. 

\begin{lemma}[Comparison Principle] \label{compa}
	Let $\lambda\in \mathbb R$, $N\geq 3$ and $\omega$ be a smooth bounded domain in $\mathbb{R}^N$ with $\overline{\omega}\subseteq \mathbb{R}^N\setminus \{0\}$. Let $b\in C^{0,\tau}(\overline{\omega})$ satisfy $b>0$ in $\omega$, where $\tau\in (0,1)$.  Assume that $g$ is a real-valued continuous function on $(0,\infty)$ such that $g(t)/t$ is increasing for $t>0$. 
	
	If $u$ and $v$ are positive $C^1(\omega)$-functions such that
	\begin{equation}\nonumber
	\left\{
	\begin{aligned}
	& -\mathbb L_\lambda (u)+b(x)\,g(u)\leq 0\leq -\mathbb L_\lambda(v)+b(x)\,g(v)\,\,\,\text{in}\,\,\, {\mathcal D}'(\omega),\\
	& \limsup\limits_{{\rm dist}(x,\partial \omega)\to 0}[u(x)-v(x)]\leq 0\,,
	\end{aligned}
	\right.
	\end{equation} 
	then $u\leq v$ in $\omega$. 
\end{lemma}

Our next result is obtained in the same way as Lemma~4.1 in \cite[Chapter 4]{Cmem}, where we take $b(x)=|x|^\theta$ for $x\in \mathbb R^N\setminus \{0\}$ and $h(t)=t^q$ for every $t\in (0,\infty)$ with $q>1$ and $\theta\in \mathbb R$. 

\begin{lemma}[{\em A priori} estimates]\label{lem01}  
	Let $r_0>0$ be such that $\overline{B_{2r_0}(0)}\subset \Omega_0$. 
	For every $q>1$ and $\lambda,\theta\in \mathbb R$, there exists a constant $C_0>0$, depending only on 
	$N,q,\lambda$ and $\theta$,  
	such that any sub-solution $u$ of \eqref{eq1} with $\Omega=\Omega_0$ satisfies
	\begin{equation}\label{in11}
	u(x)\leq C_0\, |x|^{-\Theta} 
	\quad\mbox{for all }0<|x|\leq r_0. 
	\end{equation}
\end{lemma}

\begin{proof}
	Fix $x_0\in \mathbb{R}^N$ with $0<|x_0|\leq r_0$. 
	For every $x\in B_{|x_0|/2}(x_0)$, we define 
	\begin{equation}\label{pp}
	\mathcal{P}(x):=C_0\,|x_0|^{-\Theta} \left [\zeta(x)\right]^{-\frac{2}{q-1}},\quad\mbox{where } 
	\zeta(x):=1-\left (\frac{2\,|x-x_0|}{|x_0|}\right )^2. 
	\end{equation}
	We claim that in \eqref{pp} we can take a constant $C_0>0$ that is independent of $x_0$ and $r_0$ such that 
	\begin{equation}\label{xz}
	-\mathbb L_\lambda (\mathcal P(x))+|x|^\theta \left(\mathcal P(x)\right)^q\geq 0
	\quad\mbox{for every } x \in B_{|x_0|/2}(x_0). 
	\end{equation}
	Indeed, a simple calculation shows that the inequality in \eqref{xz} is equivalent to
	\begin{equation} \label{jul}
	\frac{|x_0|^\theta}{|x|^\theta}
	\left\{
	\frac{16}{q-1}\left [N\zeta(x)+\frac{8\,(q+1)}{q-1}\,\frac{|x-x_0|^2}{|x_0|^2}\right ]+\lambda\, \frac{|x_0|^2}{|x|^2}\,\zeta^2(x) \right\}\leq C_0^{q-1}
	\end{equation}
	for every $x\in B_{|x_0|/2}(x_0)$. 
	Since $1/2\leq |x|/|x_0|\leq 3/2$ for each $x\in B_{|x_0|/2}(x_0)$, we see that the left-hand side of \eqref{jul} is bounded above by a positive constant depending only on $N,q,\lambda$ and $\theta$.

	Hence, we can find $C_0>0$ such that \eqref{xz} holds. 	
	Let $u$ be any sub-solution of \eqref{eq1} with $\Omega=\Omega_0$. From the definition of $\mathcal P$ in \eqref{pp}, we have $\mathcal P(x)\to \infty$ as ${\rm dist}\,(x,\partial B_{|x_0|/2}(x_0))\to 0$.  
	Then, by Lemma \ref{compa}, we obtain that 
	\begin{equation}
	u(x)\leq \mathcal{P}(x)\quad\mbox{for every } x\in B_{|x_0|/2}(x_0).
	\end{equation}
	In particular, for $x=x_0$ we have $u(x_0)\leq \mathcal P(x_0)=C_0\,|x_0|^{-\Theta}$. Since this inequality holds for every $0<|x_0|\leq r_0$, we conclude the proof of \eqref{in11}. 
\end{proof}

Since the constant $C_0$ in Lemma~\ref{lem01} is independent of the domain, we obtain global {\em a priori} estimates for any positive solution of \eqref{eq1}. 

\begin{corollary}[Global {\em a priori} estimates] \label{lem1} Let $\Omega=\mathbb R^N$. For every $q>1$ and $\lambda,\theta\in \mathbb R$, 
	there exists a constant $C_0>0$, depending only on 
	$N$, $q$, $\lambda$ and $\theta$,  
	such that every positive sub-solution $u$ of \eqref{eq1} satisfies
	\begin{equation}\label{in123}
	u(x)\leq C_0\, |x|^{-\Theta} 
	\quad\mbox{for all }x\in \mathbb R^N\setminus \{0\}. 
	\end{equation}
\end{corollary} 

We next state a regularity result from \cite[Lemma 4.9]{Cmem}, proved there in a more general setting. 
We recall that a positive measurable function $\phi$ defined on some interval $(0,A)$ with $A>0$ is called {\em regularly varying at zero with index}  $m\in \mathbb R$, or $\phi\in RV_m(0+)$ in short, provided that  
$$ \lim_{r\to 0^+}\frac{\phi(\xi r)}{\phi(r)}=\xi^m \quad \mbox{for every } \xi>0.$$  
When $m=0$, we say that $\phi$ is {\em slowly varying at zero}. Any positive constant is a slowly varying function at zero. Non-trivial examples of slowly varying functions at zero (defined for $r>0$ small) include 
\begin{enumerate}
	\item[(a)] the logarithm $\log (1/r)$, its iterates $\log_k (1/r)$ (defined as $\log (\log_{k-1}(1/r))$) and powers of
	$\log_k (1/r)$ for every integer $k\geq 1$;
	\item[(b)] $\exp\left(\frac{\log (1/r)}{\log \log (1/r)}\right)$;  
	\item[(c)] $\exp[(-\log r)^\nu]$ for $\nu\in (0,1)$. 
\end{enumerate} 

\begin{lemma}[A regularity result]  \label{regu}
	Let $r_0>0$ be such that $\overline{B_{4r_0}(0)}\subset \Omega_0$. Let $0\leq \delta\leq \Theta$ and $g\in RV_{-\delta}(0+)$ be a positive continuous function on $(0,4r_0)$ such that 
	$\limsup_{|x|\to 0} |x|^\Theta g(r)<\infty $. 
	If $u$ is a solution of \eqref{eq1} with $\Omega=\Omega_0$ such that, for some constant $C_1>0$, we have
	\begin{equation} \label{rafi} 0<u(x)\leq C_1 \,g(|x|)\quad \mbox{for every } 0<|x|< 2 r_0, 
	\end{equation} 
	then there exist constants $C>0$ and $\alpha\in (0,1)$ such that 
	\begin{equation} \label{eti}
	|\nabla u(x)|\leq C\,\frac{g(|x|)}{|x|}\quad \mbox{and}\quad |\nabla u(x)-\nabla u(x')|\leq C\,\frac{g(|x|)}{|x|^{1+\alpha}}\, |x-x'|^\alpha
	\end{equation}
	for every $x,x'$ in $\mathbb R^N$ satisfying $0<|x|\leq |x'|<r_0$. 
\end{lemma}

\begin{remark} \label{r-regu}
	If in Lemma~\ref{regu} we assume that \eqref{rafi} holds 
	for $g\in RV_{-\delta}(0+)$ with $\delta\leq \Theta<0$ or 
	$\delta<0\leq \Theta$, then the assertion of \eqref{eti} remains valid, subject to a slight change only in the second inequality, which should be replaced by 
	$$  |\nabla u(x)-\nabla u(x')|\leq C\,\frac{g(|x'|)}{|x|^{1+\alpha}}\, |x-x'|^\alpha
	$$
	for every $x,x'$ in $\mathbb R^N$ satisfying $0<|x|\leq |x'|<r_0$. Here, and in the first inequality of \eqref{eti}, the constant $C$ will depend on $|\delta|$ (only when $\delta<0$). The explanation for these changes is provided in Remark~4.11 of \cite[p. 34]{Cmem}.  
\end{remark}

\section{Proof of Theorem~\ref{uniq} and Theorem~\ref{mar}} \label{t31}

The claim of Theorem~\ref{mar} follows 
from Corollary~\ref{fum} and Proposition~\ref{lim-o}. 
The ingredients necessary for proving the latter are given in 
Section~\ref{motiv}.  We conclude the assertion of Theorem~\ref{uniq} in Section~\ref{sec-uni} based on Theorem~\ref{mar}. 

To simplify writing, by Case $(\mathcal M)$, we mean Case~$(\mathcal M_1)$ or Case~$(\mathcal M_2)$. 
Here, we assume that either Case $(\mathcal M)$ or Case $(\mathcal U)$ holds, that is,
\begin{enumerate}
	\item[$(\mathcal M)$] $\lambda\leq \lambda_H$ and 
	$\theta\in (-\infty,\theta_-)\cup (\theta_+,\infty)$;
	\item[$(\mathcal U)$] $\lambda>\lambda_H$ and every $\theta\in \mathbb R$.  
\end{enumerate}  

We construct refined local sub/super-solutions of \eqref{eq1} with $\Omega=B_1(0)$, which we use to fine-tune the behavior of the positive solutions of \eqref{eq1} near zero. We illustrate this point. 
In Case $(\mathcal M)$ and Case $(\mathcal U)$, we always have at our disposal the solution $U_0$ of \eqref{eq1}.
Using our super-solutions constructed in this section, jointly with the {\em a priori} estimates in Lemma~\ref{lem01}, we obtain in Proposition~\ref{lim-o} that every  (sub-)solution of \eqref{eq1} with $\Omega=\Omega_0$ satisfies 
\begin{equation} \label{hap} \limsup_{|x|\to 0} \frac{u(x)}{U_0(x)}\leq 1.\end{equation} 
Moreover, using our refined local sub/super-solutions in Cases~$(\mathcal U)$ and $(\mathcal M)$, we show that the proof of $\lim_{|x|\to 0} u(x)/U_0(x)=1$ reduces to proving 
\begin{equation} \label{lin} \liminf_{|x|\to 0} \frac{u(x)}{U_0(x)}>0. \end{equation} 


\subsection{Construction and motivation of our refined sub/super-solutions} \label{motiv}

The idea of constructing a suitable family of sub-solutions and super-solutions to obtain more precise upper and lower bound estimates near zero has been used successfully for various nonlinear elliptic equations without a Hardy potential, see for example  \cites{cc,CC2015,CD2007,CD2010}. 

In our situation, the introduction of the Hardy potential in Case $(\mathcal U)$ and Case $(\mathcal M_1)$ poses an extra difficulty when comparing an arbitrary solution $u$ of \eqref{eq1} with a super-solution (or sub-solution). Such a comparison will take place on a punctured ball, $B_{r_0}(0)\setminus \{0\}$ with $r_0\in (0,1)$ small enough such that $\overline{B_{r_0}(0)}\subset \Omega$. To apply the comparison principle in Lemma~\ref{compa}, we need to ensure that the solution $u$ is bounded above by the super-solution (and below by the sub-solution) on $\partial B_{r_0}(0)$ and also as $|x|\to 0$.

Let $\alpha>0$ and $\nu>0$ be fixed. For every $\varepsilon\in (0,1)$ and $\eta>0$, we define
\begin{equation} \label{sub-super} 
\begin{aligned}
& w_{\varepsilon,\eta}^+(x):= \left(1+ \varepsilon\right) 
U_0(x)\,|x|^{-\eta} \left( 1+\frac{|x|^\alpha}{\nu}\right)^{ \frac{1}{\sqrt{\alpha}}},\\
& w_{\varepsilon,\eta}^-(x):= \left(1- \varepsilon\right) 
U_0(x)\,|x|^{\eta} \left( 1+\frac{|x|^\alpha}{\nu}\right)^{ -\frac{1}{\sqrt{\alpha}}},
\end{aligned}
\end{equation}
for every $x\in \mathbb R^N \setminus \{0\}$, where $U_0$ is given by 
\eqref{node}.

Assuming \eqref{lin} and using the {\em a priori} estimates in Lemma~\ref{lem01}, it is clear that we get the desired control {\em near zero} by introducing along $U_0(x)$ the factor 
$|x|^\eta$ in the sub-solution $w_{\varepsilon,\eta}^-$ and the factor 
$|x|^{-\eta}$ in the super-solution $w_{\varepsilon,\eta}^+$. 

Even though for every $\varepsilon\in (0,1)$ and $\eta>0$, we find that  $(1+\varepsilon)\,U_0(x)|x|^{-\eta}$ and $(1-\varepsilon)\,U_0(x)|x|^{\eta}$ is a super-solution and sub-solution of \eqref{eq1} in $B_1(0)\setminus \{0\}$, respectively, the shortcoming of these becomes apparent when comparing them with $u$ on $\partial B_{r_0}(0)$. As we take $\eta\to 0$ and eventually $\varepsilon\to 0$, we need another degree of freedom to adjust the values of  sub/super-solutions on $\partial B_{r_0}(0)$. 

Previously, the above issue was resolved by adding to the super-solution a corrective term (itself a super-solution) such that its behavior near zero is dominated by $U_0$. But only in  
Case $(\mathcal M_2)$ this strategy can work as follows: we can add $C|x|^{-p_-}$ to the super-solution $(1+\varepsilon)\,U_0(x)|x|^{-\eta}$ (or to the solution $u$ in order to control the sub-solution $(1-\varepsilon)\,U_0(x)|x|^{\eta}$), where $C>0$ is a suitable constant depending on $u$ and $r_0$. This works well only in Case $(\mathcal M_2)$ since then $|x|^{\Theta-p_-}\to 0$ as $|x|\to 0$ and $C|x|^{-p_-}$ is a super-solution of \eqref{eq1}. 
But the above strategy does not work in Case~$(\mathcal M_1)$ or Case $(\mathcal U)$. Indeed, in Case~$(\mathcal M_1)$ we have $|x|^{\Theta-p_-}\to \infty$ as $|x|\to 0$, whereas $p_-$ is not well-defined in Case $(\mathcal U)$ when $\lambda>\lambda_H$. For this reason, we have to reshape our super-solutions and sub-solutions: we multiply  $(1+\varepsilon)\,U_0(x)|x|^{-\eta}$ by an {\em extra} factor of the form $ \left(1+|x|^\alpha/\nu\right)^{1/\sqrt{\alpha}} $ giving the super-solution $w_{\varepsilon,\eta}^+$ and correspondingly multiply $(1-\varepsilon)\,U_0(x)|x|^{\eta}$ by  $\left(1+|x|^\alpha/\nu\right)^{-1/\sqrt{\alpha}}$ to yield the sub-solution $w_{\varepsilon,\eta}^-$,  
where $\alpha>0$ is fixed suitably small, depending only on $N,q,\theta$ and $\lambda$, while $\nu>0$ is arbitrary.   
The verification that $w_{\varepsilon,\eta}^+$ and $w_{\varepsilon,\eta}^-$ is a super-solution and sub-solution of \eqref{eq1} in $B_1(0)\setminus \{0\}$, respectively, is done in Lemma~\ref{sun-s}. We can now  choose $\nu>0$ small, depending only on $r_0,u,N,q,\theta$ and $\lambda$, such that  
$w_{\varepsilon,\eta}^{-}\leq u\leq w_{\varepsilon,\eta}^+$ on $ \partial B_{r_0}(0)$ for every $\varepsilon\in (0,1)$ and $\eta>0$ small. By the comparison principle in Lemma~\ref{compa}, we conclude \eqref{kop}.

We next proceed with the details. 

\begin{lemma} \label{sun-s} Assume Case $(\mathcal M)$ or Case $(\mathcal U)$. Fix $\alpha>0$ small, depending only on $N,q,\theta$ and $\lambda$. Let $\nu>0$ be arbitrary. For every $\varepsilon\in (0,1)$, there exists $\eta_0=\eta_0(\varepsilon,N,q,\theta,\lambda)>0$ such that 
	\begin{equation} \label{ulm} -\mathbb L_\lambda(w_{\varepsilon,\eta}^+)+|x|^\theta (w_{\varepsilon,\eta}^+)^q \geq 0\quad \mbox{and} 
	\quad -\mathbb L_\lambda(w_{\varepsilon,\eta}^-)+
	|x|^\theta (w_{\varepsilon,\eta}^-)^q \leq 0
	\end{equation}
	in $B_1(0)\setminus \{0\}$, for every $\eta\in (0,\eta_0)$. 
\end{lemma}

\begin{proof} Let $\alpha>0$. 
	For every $t\geq 0$ and $\eta>0$, we define 
	\begin{equation} \label{geta} G_\eta^\pm (t):=  \left (1+t\right )^{-2\mp\frac{(q-1)}{\sqrt{\alpha}}} \left(A_\eta^\pm \,t^2+B_\eta^\pm\, t+C_\eta^\pm\right), 
	\end{equation} where 
	$A_\eta^\pm$, $B_\eta^\pm$ and $C_\eta^\pm$ are given by 
	\begin{equation}\nonumber
	\left\{\begin{aligned}
	& A_\eta^\pm:= 1\pm\frac{\sqrt{\alpha}\left(N-2-2\,\Theta\pm\sqrt{\alpha}\right)}{\ell}\mp\frac{\eta\left(N-2-2\,\Theta\mp\eta\pm 2\sqrt{\alpha}\right)}{\ell},\\
	& B_\eta^\pm:=2\pm\frac{\sqrt{\alpha}\left(N-2-2\,\Theta+\alpha\right)}{\ell}\mp\frac{2\,\eta\left(N-2-2\,\Theta\mp\eta\pm\sqrt{\alpha}\right)}{\ell},\\
	&	C_\eta^\pm:=1\mp\frac{\eta\left(N-2-2\,\Theta\mp\eta\right)}{\ell}\,.
	\end{aligned} \right.
	\end{equation}
	From the definition of $G_\eta^\pm$ in \eqref{geta}, we find that 
	\begin{equation}\nonumber
	\frac{d}{dt}	G_\eta^\pm(t)=\mp\frac{(q-1)}{\sqrt{\alpha}}\left(1+t\right)^{-3\mp\frac{(q-1)}{\sqrt{\alpha}}}\left(A_\eta^\pm\, t^2+\widetilde{B}_\eta^\pm \,t+\widetilde{C}_\eta^\pm\right),
	\end{equation}
	for every $t>0$, where  $\widetilde{B}_\eta^\pm$ and $\widetilde{C}_\eta^\pm$ are defined by
	\begin{equation}\nonumber
	\left\{\begin{aligned}
	& \widetilde{B}_\eta^\pm:=\left(1\pm \frac{\sqrt{\alpha}}{q-1}\right) B_\eta^\pm \mp \frac{2\sqrt{\alpha}}{q-1}\,A_\eta^\pm\\
	& \widetilde{C}_\eta^\pm:=\left(1\pm \frac{2\,\sqrt{\alpha}}{q-1}\right) C_\eta^\pm \mp \frac{\sqrt{\alpha}}{q-1}\,B_\eta^\pm.
	\end{aligned}\right.
	\end{equation}
	We choose $\alpha>0$ small enough, depending only on $N,q,\theta$ and $\lambda$, such that 
	$$ \lim_{\eta\to 0}  A_\eta^\pm>0,\quad 
	\lim_{\eta\to 0} \widetilde  B_\eta^\pm>0 \quad
	\mbox{and}\quad
	\lim_{\eta\to 0}  \widetilde C_\eta^\pm>0.
	$$
	Hence, there exists $\eta_1=\eta_1(N,q,\theta,\lambda)>0$ such that 
	$ A_\eta^\pm$, $\widetilde B_\eta^\pm$ and $\widetilde C_\eta^\pm$ are {\em all positive} for every $\eta\in (0,\eta_1)$. Therefore, $G_\eta^+$ is decreasing on $(0,\infty)$, whereas $G_\eta^-$ is increasing on $(0,\infty)$, leading to 
	$$ \sup_{t\in (0,\infty)} G_\eta^+(t)=G_\eta^+(0)=C_\eta^+\quad 
	\mbox{and}\quad \inf_{t\in (0,\infty)} G_\eta^-(t)=G_\eta^-(0)=C_\eta^-. 
	$$
	By direct computations, we observe that \eqref{ulm} holds if and only if 
	\begin{equation}\label{ssbb}
	|x|^{\eta\left(q-1\right)} G_\eta^+(|x|^\alpha/\nu)\leq (1+\varepsilon)^{q-1}\ \mbox{and} \ 
	|x|^{-\eta\left(q-1\right)} G_\eta^-(|x|^\alpha/\nu)\geq (1-\varepsilon)^{q-1}
	\end{equation}
	for every $|x|\in (0,1)$. Since $\lim_{\eta\to 0} C_\eta^\pm=1$, we observe that there exists $\eta_0\in (0,\eta_1)$ with $\eta_0$ depending on $\varepsilon,N,q,\theta$ and $\lambda$ such that 
	$ C_\eta^+\leq \left(1+\varepsilon\right)^{q-1}$ and  
	$C_\eta^-\geq \left(1-\varepsilon\right)^{q-1}$ for every $\eta\in (0,\eta_0)$. 
	Thus, \eqref{ssbb} is satisfied for every $0<|x|<1$
	and all $\eta\in (0,\eta_0)$. This finishes the proof. 
\end{proof}

\begin{proposition} \label{lim-o}	
	In Case $(\mathcal M)$ and Case~$(\mathcal U)$, every positive solution of \eqref{eq1} with $\Omega=\Omega_0$ satisfies \eqref{hap}. 
	In addition, we have   
	\begin{equation} \label{limin}  \liminf_{|x|\to 0} \frac{u(x)}{U_0(x)}>0\quad \mbox{if and only if}\quad \lim_{|x|\to 0} \frac{u(x)}{U_0(x)}=1.
	\end{equation} 
\end{proposition}

\begin{proof} Let $r_0\in (0,1)$ be such that 
	$\overline{B_{r_0}(0)}\subset \Omega$. Fix $\alpha>0$ as in Lemma~\ref{sun-s}. Let $u$ be a positive solution of \eqref{eq1} in $\Omega\setminus \{0\}$. 
	We choose $\nu=\nu(r_0,u,N,q,\theta,\lambda)>0$ small 
	such that  the following two inequalities hold
	\begin{equation}\nonumber
	\left\{\begin{aligned}
	& U_0(r_0) \left(1+\frac{r_0^\alpha}{\nu}\right)^{-\frac{1}{\sqrt{\alpha}}} \leq \min_{x\in \partial B_{r_0}(0)} u(x)\\
	& U_0(r_0)	\left(1+\frac{r_0^\alpha}{\nu}\right)^{\frac{1}{\sqrt{\alpha}}}\geq \max_{x\in \partial B_{r_0}(0)} u(x)  .
	\end{aligned}\right.
	\end{equation}
	Fix $\varepsilon\in (0,1)$ arbitrary. Let $\eta_0$ be given by Lemma~\ref{sun-s}. 
	Our choice of $\nu>0$ ensures that 
	\begin{equation} \label{li1} w_{\varepsilon,\eta}^{-}(x)\leq u(x)\leq w_{\varepsilon,\eta}^+(x) \ \mbox{for every } x\in \partial B_{r_0}(0)\ \mbox{and every } \eta\in (0,\eta_0),\end{equation} where $w_{\varepsilon,\eta}^{+}$ and $w_{\varepsilon,\eta}^{-}$ are defined in \eqref{sub-super}. Using 
	Lemma~\ref{lem01}, we find that 
	\begin{equation} \label{lit2} 
	\lim_{|x|\to 0}\frac{u(x)}{w_{\varepsilon,\eta}^{+}(x)}=0. 
	\end{equation}
	Moreover, if $\liminf_{|x|\to 0} u(x)/U_0(x)>0$, then we find in addition that  
	\begin{equation} \label{huv}  \lim_{|x|\to 0} \frac{w_{\varepsilon,\eta}^{-}(x)}{u(x)}=0.
	\end{equation}
	In view of \eqref{ulm} and \eqref{li1}--\eqref{huv}, by the comparison principle in Lemma~\ref{compa}, we infer that
	\begin{equation}\label{sad}
	w_{\varepsilon,\eta}^-(x)\leq u(x)\leq w_{\eta,\varepsilon}^+(x)\quad\text{for every } 0<|x|\leq r_0 \ \mbox{ and all } \eta\in (0,\eta_0).  
	\end{equation}
	For every $x\in B_{r_0}(0)\setminus \{0\}$ fixed, by letting $\eta\rightarrow 0$ in  \eqref{sad}, we arrive at 
	$$ \left(1-\varepsilon\right)   \left( 1+\frac{|x|^\alpha}{\nu}\right)^{ -\frac{1}{\sqrt{\alpha}}}\leq \frac{u(x)}{U_0(x)}\leq \left(1+\varepsilon\right)   \left( 1+\frac{|x|^\alpha}{\nu}\right)^{ \frac{1}{\sqrt{\alpha}}}.
	$$
	Thus, for every $\varepsilon\in (0,1)$, it follows that 
	$$ 1-\varepsilon\leq \liminf_{|x|\to 0}  \frac{u(x)}{U_0(x)}\leq \limsup_{|x|\to 0}\frac{u(x)}{U_0(x)}\leq 1+\varepsilon.
	$$
	Hence, by passing to the limit $\varepsilon\to 0$, we conclude that $\lim_{|x|\to 0} u(x)/U_0(x)=1$ as desired. This completes the proof.  
\end{proof}

\begin{remark} In the framework of Proposition~\ref{lim-o}, every solution $u$ of \eqref{eq1} satisfies \eqref{hap}. On the other hand, to prove that $\lim_{|x|\to 0} u(x)/U_0(x)=1$, the hypothesis $\liminf_{|x|\to 0} u(x)/U_0(x)>0$ is necessary and we cannot dispense with in Case $(\mathcal M_2)$. To see this, we draw attention to Case $(\mathcal M_2)$ in Theorem~\ref{th-cd} when a solution $u$ of \eqref{eq1} may satisfy {\bf (A)} $\lim_{|x|\to 0}|x|^{p_-} u(x)\in (0,\infty)$ or {\bf (B)} $\lim_{|x|\to 0} u(x)/\Phi_\lambda^+(x)=\gamma\in (0,\infty)$ (and then $\lim_{|x|\to 0} u(x)/U_0(x)=0$). Theorem~\ref{est} shows that there exist solutions for \eqref{eq1} in each of the situations outlined in Theorem~\ref{th-cd}.    
\end{remark}

\subsection{Proof of Theorem~\ref{uniq}} \label{sec-uni}

Let $\Omega=\mathbb R^N$. We show that $U_0$ in \eqref{node} is the only solution of \eqref{eq1} in Case~$({\mathcal U})$. 
Let $u$ be a solution of \eqref{eq1}. Let $\varepsilon\in (0,1)$ be arbitrary. 
By Theorem~\ref{mar}, $u$ satisfies \eqref{kop}. 
Then, using the Kelvin transform (see Section~\ref{mainr}), we obtain \eqref{kop2}. 
Hence, there exist  $R_\varepsilon>r_\varepsilon>0$ such that 
\begin{equation}\label{blai} \left(1-\varepsilon\right)U_0(x)\leq u(x)\leq \left(1+\varepsilon\right)U_0(x)\quad \mbox{for every }|x|\in (0,r_\varepsilon]\cup [R_\varepsilon,\infty).\end{equation}
Since $U_0$ is a positive solution of \eqref{eq1}, we find that 
$$ -\mathbb L_\lambda ((1+\varepsilon)\,U_0)+|x|^\theta (1+\varepsilon)^q \,U_0^q\geq 0\quad \mbox{in } \mathbb R^N\setminus \{0\}$$ and, similarly, 
$$ -\mathbb L_\lambda ((1-\varepsilon)\,U_0)+|x|^\theta (1-\varepsilon)^q \,U_0^q\leq 0\quad \mbox{in } \mathbb R^N\setminus \{0\}.$$
Hence, the comparison principle in Lemma~\ref{compa} gives that 
the inequalities in \eqref{blai} hold for every $x\in \mathbb R^N\setminus \{0\}$. By letting $\varepsilon\to 0$, we arrive at $u\equiv U_0$ in $\mathbb R^N\setminus \{0\}$. This ends the proof of Theorem~\ref{uniq}.

\section{Proof of Theorem~\ref{est}} \label{estp}

Our aim is to prove the assertions of Theorem~\ref{est} on problem \eqref{boun}, namely, 
\begin{equation} \label{boun1} 
\left\{\begin{aligned}
& -\Delta u-\frac{\lambda}{|x|^2}u+|x|^{\theta}u^q=0\ \  \mbox{in }\Omega\setminus \{0\},&\\
& u=h\geq 0\ \  \mbox{on } \partial\Omega,\quad\ u>0 \ \ \mbox{in } \Omega\setminus\{0\},&
\end{aligned}\right.
\end{equation} 
where throughout this section, $\Omega\subseteq \mathbb R^N$ is a smooth bounded domain containing zero. 
In Lemma~\ref{nonz} we establish the first three statements in Theorem~\ref{est}, whereas the last one regarding Case $(\mathcal M_2)$ is proved separately in Lemma~\ref{m2c}.  

\begin{lemma} \label{nonz}  
	Suppose that $h\in C(\partial \Omega)$ is a non-negative 
	function. 
	\begin{enumerate}
		\item[{\rm (1)}] Let Case $(\mathcal U)$ hold. Then, there exists a unique solution $u_h$ of problem \eqref{boun1}. Moreover, if 
		$\Theta<(N-2)/2$ and $h\equiv 0$, then 
		$u_h(x)/|x|$ and $|x|^{\theta+1} u_h^q$ belong to $ L^2(\Omega)$, $u_h\in H^1_0(\Omega)$ and,
		for every $\varphi\in H_0^1(\Omega)$,   
		\begin{equation} 
		\label{dist}
		\int_\Omega \nabla u_h\cdot \nabla \varphi\,dx-\int_\Omega \frac{\lambda}{|x|^2}u_h\,\varphi\,dx+\int_\Omega
		|x|^{\theta}u_h^q\,\varphi\,dx=0.
		\end{equation} 
		\item[{\rm (2)}] 
		Assume Case $(\mathcal M_1)$ or Case $(\mathcal N)$. If $h\not\equiv 0$ on $\partial \Omega$, then problem \eqref{boun1} has a unique solution $u_h$. 

		\item[{\rm (3)}] If $h\equiv 0 $ on $\partial\Omega$, then \eqref{boun1} has no solutions in Case $(\mathcal M_1)$ and Case $(\mathcal N)$. 
		
	\end{enumerate}
\end{lemma}

\begin{proof} We divide the proof into four steps.
	
	\vspace{0.2cm}
	\noindent 	{\bf Step~1.} 
	In Cases $(\mathcal U)$, $(\mathcal M_1)$ and $(\mathcal N)$, there is at most one solution of \eqref{boun1}. 
	
	\vspace{0.2cm}
	\noindent	{\em Proof of Step~1.} 
	We show that any two solutions $u_h$ and $U_h$ of \eqref{boun1}  coincide. 	
	
	In Cases $(\mathcal U)$ and  $(\mathcal M_1)$, we derive from Theorem~\ref{mar}, that  $u_h(x)/U_h(x)\rightarrow 1$ as $|x|\rightarrow 0$. Since $u_h=U_h=h$ on $\partial\Omega$, the comparison principle in Lemma~\ref{compa} yields that, for every $\varepsilon\in (0,1)$, 
	$$ (1-\varepsilon)\, U_h\leq u_h\leq (1+\varepsilon)\,U_h
	\quad \mbox{in } \Omega\setminus \{0\}.$$  Thus, by passing to the limit with $\varepsilon\to 0$, we arrive at $u_h=U_h$ in $\Omega\setminus \{0\}$. 
	
	In Case $(\mathcal N)$ by Theorem~\ref{th-cd} we have  $\lim_{|x|\to 0} u(x)/\Phi_\lambda^+(x)=0$ with $u=u_h$ and $u=U_h$. This means that for every $\varepsilon>0$, there exists $r_\varepsilon>0$ small such that 
	$u_h(x)\leq \varepsilon \,\Phi_\lambda^+(x)$ for every $0<|x|\leq r_\varepsilon$. By Lemma~\ref{compa}, it follows that $u_h\leq \varepsilon \,\Phi_\lambda^++U_h$ in $\Omega\setminus \{0\}$. By letting $\varepsilon\to 0$ and then interchanging $u_h$ and $U_h$, we arrive at
	$u_h=U_h$ in $\Omega\setminus\{0\}$ as desired.  
	
	\vspace{0.2cm}
	\noindent 	{\bf Step~2.} 
	Problem \eqref{boun1} has a solution $u_h$ 
	in Case $(\mathcal U)$ and, moreover, if $h\not\equiv 0$ also in Cases $(\mathcal M_1)$ and $(\mathcal N)$. 
	
	\vspace{0.2cm}
	\noindent  {\em Proof of Step~2.}
	We will obtain a solution $u_h$ of \eqref{boun1} as the limit $k\to \infty$ of a non-increasing sequence $\{u_{h,k}\}_{k\geq k_0}$ of solutions to boundary value problems (see \eqref{hats})
	on approximate domains $\Omega_k:=\Omega\setminus \overline{B_{1/k}(0)}$ for $k\geq k_0$ large.

	In Cases $(\mathcal U)$ and $(\mathcal M_1)$, we know from Theorem~\ref{mar} that, whenever it exists, a solution of \eqref{boun1} satisfies 
	$\lim_{|x|\to 0} |x|^\Theta u(x)=\ell^{1/(q-1)}$. This provides the inspiration for taking the boundary value problems in \eqref{hats}. It is useful to    
	remark that if $C>0$ is large enough, then $C\,|x|^{-\Theta}$ is always a super-solution of \eqref{eq1} in $\mathbb R^N\setminus \{0\}$. In Cases $(\mathcal U)$, $(\mathcal M_1)$ and $(\mathcal M_2)$, we need only choose $C\geq \ell^{1/(q-1)}$. In Case $(\mathcal N)$, we can take any $C>0$ since $\ell\leq 0$. 
	
	If necessary, we increase $C>0$ to ensure that $C\geq \max_{x\in \partial\Omega}|x|^\Theta h(x)$. 
	Fix $\delta>0$ small such that $\overline{B_\delta(0)}\subset \Omega$. Let $k_0$ be a positive integer such that $1/k_0<\delta$. 
	Then, for every integer $k\geq k_0$, 
	the following boundary value problem
	\begin{equation} \label{hats}
	\left\{\begin{aligned}
	& -\Delta u-\frac{\lambda}{|x|^2}u+|x|^{\theta}u^q=0&& \mbox{in }\Omega_k:=\Omega\setminus \overline{B_{1/k}(0)},&\\
	& u=h &&  \mbox{on } \partial\Omega,&\\
	& u(x)=C\, |x|^{-\Theta} && \mbox{for every  }|x|=1/k,&\\
	& u>0 && \mbox{in } \Omega_k
	\end{aligned} \right.\end{equation}  
	has a unique solution $u_{h,k}\in C^2(\Omega_k)\cap C(\overline{\Omega_k})$. This assertion is true for all $\lambda,\theta\in \mathbb R$ and $q>1$. The existence of a non-negative solution $u_{h,k}$ follows from Theorem 15.18 in Gilbarg and Trudinger \cite{GT1983}, whereas the strong maximum principle (see, for example,  Theorem~2.5.1 in \cite{pucci}) yields the positivity of $u_{h,k}$ in $\Omega_k$.  
	The uniqueness of $u_{h,k}$ is a consequence of Lemma~\ref{compa}. 
	Moreover, with our choice of $C$, we obtain that $$ u_{h,k+1}\leq u_{h,k}\leq C\,|x|^{-\Theta}\quad \mbox{in } \Omega_k\ \  \text{for every } k\geq k_0.$$
	Using Lemma~\ref{regu} (see also Remark~\ref{r-regu}) and a standard argument, we get that, up to a subsequence, $u_{h,k}\rightarrow u_h$ in $C_{\text{loc}}^{1}(\Omega\setminus \{0\})$ as $k\to \infty$, where $u_h$ is a non-negative solution of \eqref{boun1}. 
	
	It remains to prove that $u_h>0$ in $\Omega\setminus\{0\}$. We treat Case $(\mathcal U)$ separately from Case $(\mathcal M_1)$ and Case $(\mathcal N)$. 
	
	$\bullet$ In Cases $(\mathcal M_1)$ and $(\mathcal N)$  
	we assume that $h\not\equiv 0$ on $\partial\Omega$. Then, by the strong maximum principle, we conclude that $u_h$ is positive in $\Omega\setminus\{0\}$. 
	
	$\bullet$ In Case $(\mathcal U)$ our argument works for any non-negative function $h\in C(\partial\Omega)$ since we have Lemma~\ref{con2} at our disposal. More precisely, for fixed $c\in (0,c_\alpha)$ as in Lemma~\ref{con2} and $\delta>0$ chosen above, we define $z_\delta$ as in \eqref{zde}. Since \eqref{da2} holds, by Lemma~\ref{compa}, we derive that 
	\begin{equation} \label{loi} u_{h,k}(x)\geq z_\delta(x) \quad \mbox{for every } 1/k\leq |x|\leq \delta\ \mbox{and all }k\geq k_0. 
	\end{equation}  
	Thus, by letting $k\to \infty$ in \eqref{loi}, we find that $u_h(x)\geq z_\delta(x)>0$ for every $0<|x|<\delta$. This gives that $\liminf_{|x|\to 0} u_h(x)/U_0(x)>0$. 
	By the strong maximum principle, we have $u_h>0$ in $\Omega\setminus\{0\}$.  
	This concludes Step~2. 
	
	\vspace{0.2cm}
	\noindent 	{\bf Step~3.}
	If 
	$\Theta<(N-2)/2$ in Case $(\mathcal U)$ and $h=0$, then 
	$u_h(x)/|x|$ and $|x|^{\theta+1} u_h^q$ belong to $ L^2(\Omega)$, $u_h\in H^1_0(\Omega)$ and \eqref{dist} holds. 
	
	\vspace{0.2cm}
	\noindent 	{\em Proof of Step~3.}
	Let   
	$\Theta<(N-2)/2$ in Case $(\mathcal U)$. 
	By Theorem~\ref{mar}, $u_h$ satisfies \eqref{kop} and thus 
	$u_h(x)/|x|$ and $|x|^{\theta+1} (u_h(x))^q$ belong to $ L^2(\Omega)$. Hence, using that $\varphi(x)/|x|\in L^2(\Omega)$ for every 
	$\varphi\in H_0^1(\Omega)$, we see that $u_h(x)\,\varphi(x)/|x|^2$ and $|x|^\theta\,(u_h(x))^q\,\varphi(x)$ belong to $L^1(\Omega)$. 
	Using Lemma~\ref{regu} when $0\leq \Theta<(N-2)/2$ and Remark~\ref{r-regu} when $\Theta<0$ (where $g(r)=r^{\Theta} $ for $r>0$), we obtain 
	\begin{equation} \label{grad} |\nabla u_h(x)|\leq C |x|^{-\Theta-1}\quad \mbox{for every } 0<|x|<r_0,\end{equation}
	where $r_0>0$ is small. This implies that 
	$u_h\in H^1_{\rm loc}(\Omega)$. 
	For every $\varepsilon\in (0,1)$ small, 
	let $w_\varepsilon$ be a non-decreasing and smooth function on $(0,\infty)$ such that 
	$$ \left\{\begin{aligned}
	& w_\varepsilon=0&& \mbox{on }(0,\varepsilon],&\\
	& 0<w_\varepsilon(r)<1 && \mbox{for every }r\in (\varepsilon,2\varepsilon),&\\ 
	& w_\varepsilon=1 && \mbox{on } [2\varepsilon,\infty).&
	\end{aligned}\right.
	$$
	
	Let $ \varphi\in C_c^1(\Omega)$ be arbitrary. Using $\varphi\,w_\varepsilon\in C^1_c(\Omega\setminus\{0\})$ as a test function in the equation \eqref{eq1} satisfied by $u_h$, we deduce that 
	\begin{equation} 
	\label{disto}
	\int_\Omega w_\varepsilon\, \nabla u_h\cdot \nabla \varphi\,dx-\int_\Omega \frac{\lambda}{|x|^2}u_h\,\varphi\,w_\varepsilon\,dx+\int_\Omega
	|x|^{\theta}u_h^q\,\varphi\,w_\varepsilon\,dx=-J_\varepsilon,
	\end{equation} 
	where for every $\varepsilon>0$, we define $J_\varepsilon$ as follows
	$$ J_\varepsilon:=
	\int_\Omega \varphi\,\nabla u_h\cdot \nabla w_\varepsilon\,dx=
	\int_{\{\varepsilon<|x|<2\varepsilon\}} \varphi(x)\,
	\frac{w_\varepsilon'(|x|)}{|x|}\,\nabla u_h(x)\cdot x\,dx.
	$$  
	With the estimate in \eqref{grad} and relying on the assumption $\Theta<(N-2)/2$, it is easy to see that $J_\varepsilon\to 0$ as $\varepsilon\to 0$. Hence, by letting $\varepsilon\to 0$ in \eqref{disto}, we infer that \eqref{dist} holds for every $ \varphi\in C_c^1(\Omega)$. 
	
	\vspace{0.2cm}
	We now assume that $h=0$ in \eqref{boun1}. Instead of $u_h$, we use the notation $u_0$. We claim that $u_0\in H_0^1(\Omega)$ and \eqref{dist} holds  
	for each $\varphi\in H_0^1(\Omega)$. 
	
	Let $\delta>0$ be small such that $\overline{B_\delta(0)}\subset \Omega$. 
	We define 
	$\omega=\Omega\setminus \overline{B_\delta(0)}$. 
	By the classical trace theory, there exists a function 
	$f\in H^1(\omega)\cap C(\overline \omega)$ such that $f=u_0$ on $\partial \omega$. By the classical regularity theory, we infer that $u_0\in H^1(\omega)$. This proves that $u_0\in H^1(\Omega)$. Since $u_0=0$ on $\partial\Omega$, we conclude that $u_0\in H_0^1(\Omega)$. 
	
	Now, for each $\varphi\in H_0^1(\Omega)$, there exists a sequence $\{\varphi_n\}_{n\geq 1}$ in $C^1_c(\Omega)$ such that $\varphi_n\to \varphi$ in $H^1(\Omega)$ as $n\to \infty$. Then, by the Hardy inequality in \eqref{har}, 
	$$ \int_\Omega \frac{(\varphi_n-\varphi)^2}{|x|^2}\,dx\to  0\quad \mbox{as } n\to \infty. 
	$$
	Hence, by H\"older's inequality, as $n\to \infty$, we find that   
	\begin{equation} \label{sec} \int_\Omega \frac{u_0}{|x|^2}\,\varphi_n\,dx\to \int_\Omega \frac{u_0}{|x|^2}\,\varphi\,dx\,\, \mbox{and}\,\,  \int_\Omega
	|x|^{\theta}u_0^q\,\varphi_n\,dx\to \int_\Omega
	|x|^{\theta}u_0^q\,\varphi\,dx.\end{equation} 
	For the second limit in \eqref{sec} we also use that $\sup_{x\in \Omega}|x|^{\theta+2} (u_0(x))^{q-1}<\infty$ in view of \eqref{kop} and $u_0=0$ on $\partial\Omega$. 
	Since \eqref{dist} holds with $\varphi_n$ instead of $\varphi$, by letting $n\to \infty$, we extend  \eqref{dist} to every $\varphi\in H_0^1(\Omega)$.  This finishes Step~3. 
	
	\vspace{0.2cm}
	\noindent 	{\bf Step~4.} If $h\equiv 0$, then \eqref{boun1} has no solution in Case $(\mathcal M_1)$ and Case $(\mathcal N)$.  
	
	\vspace{0.2cm}
	\noindent	{\em Proof of Step~4.}	
	Suppose that $u$ is a solution of \eqref{boun1} with $h=0$ in Case $(\mathcal M_1)$ or Case~$(\mathcal N)$. Then, in Case $(\mathcal M_1)$ we have $\Theta<p_-$ so that  Theorem~\ref{mar} implies that $\lim_{|x|\to 0} |x|^{p_-} \,u(x)=0$. Hence, for every $\varepsilon>0$, we obtain that $u(x)\leq \varepsilon |x|^{-p_-}$ for $|x|>0$ close to zero and for every $x\in \partial\Omega$. The comparison principle (in Lemma~\ref{compa}) gives that $0<u(x)\leq \varepsilon |x|^{-p_-}$ for every $x\in \Omega\setminus\{0\}$. By letting $\varepsilon\to 0$, we arrive at $u\equiv 0$ in $\Omega\setminus\{0\}$, which is a contradiction. 
	
	The same argument applies in Case $(\mathcal N)$ whenever $\theta=\theta_-\leq \theta_+$ since from Theorem~\ref{th-cd}, we have  $\lim_{|x|\to 0} |x|^{p_-} \,u(x)=0$. In the remaining situation of Case $(\mathcal N)$, namely, when $\theta_-<\theta\leq \theta_+$ (relevant for  $\lambda<\lambda_H$), we use that $\lim_{|x|\to 0} |x|^{p_+} u(x)=0$. The above ideas work with $p_+$ instead of $p_-$ so that $0<u(x)\leq \varepsilon |x|^{-p_+}$ for every $x\in \Omega\setminus \{0\}$ and every $\varepsilon>0$. Hence, we again obtain a contradiction by letting $\varepsilon\to 0$. 
	This completes the proof of Step~4 and, hence, of Lemma~\ref{nonz}. 
\end{proof}

\begin{remark} \label{arg}
	Despite the similarity revealed in Theorem~\ref{mar} between Case $(\mathcal U)$ and Case $(\mathcal M_1)$, the difference between these comes to the fore when considering the problem \eqref{boun1} with $h=0$, which has no solutions in Case $(\mathcal M_1)$. We give an alternative proof  using the Hardy inequality. 
	Assume by contradiction that problem \eqref{boun1} with $h=0$ on $\partial\Omega$ has a solution $u_0$ in Case~$(\mathcal M_1)$. 
	Observe that $\Theta<(N-2)/2$ holds in Case~$(\mathcal M_1)$. In view of Theorem~\ref{mar}, the argument used in Step~3 for Case $(\mathcal U)$ applies to Case $(\mathcal M_1)$. 
	Hence, $u_0\in H_0^1(\Omega)$ and by taking $\varphi=u_0$ in \eqref{dist}, we get
	\begin{equation} \label{apt}  
	\int_\Omega |\nabla u_0|^2\,dx-\int_\Omega \frac{\lambda}{|x|^2}u_0^2\,dx+\int_\Omega
	|x|^{\theta}u_0^{q+1}\,dx=0.
	\end{equation}
	Since $u_0>0$ in $\Omega\setminus \{0\}$ and $\lambda\leq \lambda_H$ in Case $(\mathcal M_1)$, the Hardy inequality in \eqref{har} yields a contradiction. This proof breaks down in Case $(\mathcal U)$ 
	when $\lambda>\lambda_H$. 
\end{remark}

\begin{lemma} \label{m2c}
	Let 
	$h\in C(\partial \Omega)$ be a non-negative 
	function. Assume Case $(\mathcal M_2)$. Then, for each $\gamma\in (0,\infty]$ (also for $\gamma=0$ if $h\not\equiv 0$), problem \eqref{boun1}, subject to 
	\begin{equation} \label{var} \lim_{|x|\to 0} \frac{u(x)}{\Phi_\lambda^+(x)}=\gamma 
	\end{equation} has a unique solution $u_h^{(\gamma)}$. For $\gamma=\infty$, we have  
	$\lim_{|x|\rightarrow 0}u_h^{(\gamma)}(x)/U_0(x)= 1$. 
	\begin{enumerate} 
		\item[{\rm (1)}] If $h\not\equiv 0$ on $\partial\Omega$, then the set of all solutions of \eqref{boun1} is $\{u_h^{(\gamma)}:\ 0\leq \gamma\leq \infty\}$, where 
		for $\gamma=0$ we have $\lim_{|x|\to 0} |x|^{p_-}u_h^{(\gamma)}(x)\in (0,\infty)$.
		\item[{\rm (2)}]  
		If $h\equiv 0$ on $\partial\Omega$, then
		the set of all solutions of \eqref{boun1} is $\{u_h^{(\gamma)}:\ 0<\gamma\leq \infty\}$. 
	\end{enumerate}
\end{lemma}

\begin{proof} We first show that for  $\gamma=\infty$, 
	problem \eqref{boun1}, subject to \eqref{var}, has a unique solution 
	$u_h^{(\gamma)}$. By Theorem~\ref{th-cd}, any such solution
	$u_h^{(\gamma)}$ must satisfy $\lim_{|x|\to 0} u_h^{(\gamma)}(x)/U_0(x)=1$.  
	To construct $u_h^{(\gamma)}$, we proceed exactly like in Step~2 in the proof of Lemma~\ref{nonz} in Case $(\mathcal U)$ replacing $u_h$ by $u_h^{(\gamma)}$. Thus, we obtain a solution $u_h^{(\gamma)}$ of \eqref{boun1} satisfying $u_h^{(\gamma)}(x)\geq z_\delta(x)$ for every $0<|x|<\delta$. Recall that $z_\delta$ is given by Lemma~\ref{con2}. It follows that  
	$\liminf_{|x|\to 0} u_h^{(\gamma)}(x)/U_0(x)>0$. Then, by Proposition~\ref{lim-o}, we conclude that $\lim_{|x|\to 0} u_h^{(\gamma)}(x)/U_0(x)=1$ as desired. The uniqueness of such a solution is a simple consequence of Lemma~\ref{compa} (see Step~1 in the proof of Lemma~\ref{nonz}). 
	
	\vspace{0.2cm}
	Let $\gamma\in (0,\infty)$ be arbitrary. We prove that \eqref{boun1}, subject to \eqref{var} has a solution $u_h^{(\gamma)}$, which is unique by Lemma~\ref{compa}.  
	To construct $u_h^{(\gamma)}$, we follow the argument in the proof of \cite[Lemma~5.6]{Cmem}. For the reader's convenience, we give the details. 
	From Lemma~\ref{compa} and \cite[Propositions~3.1(c) and ~3.4(c)]{Cmem}, there exists a unique positive (radial) solution $\underline{u}_\gamma$ for the  problem
	\begin{equation} \label{ati}
	\left\{\begin{aligned}
	& -\Delta u-\frac{\lambda}{|x|^2}u+|x|^{\theta}u^q=0 &&  \mbox{in } B_1(0)\setminus\{0\},&\\
	& u=1 && \mbox{on } \partial B_1(0),&\\
	& \lim_{|x|\to 0} \frac{u(x)}{\Phi_\lambda^+(x)}=\gamma.&
	\end{aligned} \right.\end{equation}   
	Fix $\delta\in (0,1)$ such that $\overline{B_\delta(0)}\subset \Omega$. 
	Choose a constant $C>0$  large such that 
	\begin{equation} \label{top} C \Phi_\lambda^- \geq h\ \mbox{on } \partial \Omega \quad \mbox{and}\quad
	C\,\Phi_\lambda^-\geq \underline{u}_\gamma\ \ \mbox{on }\  \partial B_\delta(0).\end{equation}     
	Based on the second inequality in \eqref{top}, we derive from Lemma~\ref{compa} that 
	\begin{equation} \label{nul} \underline{u}_\gamma\leq \gamma\,\Phi_\lambda^++C\,\Phi_\lambda^-\quad \mbox{in } B_\delta(0).\end{equation}  
	
	Let $k_0>1/\delta$. For
	each integer $k\geq k_0$, the boundary value problem
	\begin{equation} \label{hati}
	\left\{\begin{aligned}
	& -\Delta u-\frac{\lambda}{|x|^2}u+|x|^{\theta}u^q=0&& \mbox{in }\Omega_k:=\Omega\setminus \overline{B_{1/k}(0)},&\\
	& u=h && \partial\Omega,\\
	& u=\gamma\, \Phi_\lambda^++C\,\Phi_\lambda^- && \mbox{on }\partial B_{1/k}(0),&\\
	& u>0 && \mbox{in } \Omega_k
	\end{aligned} \right.\end{equation} 
	has a unique solution $u^{(\gamma)}_{h,k}\in C^2(\Omega_k)\cap C(\overline{\Omega_k})$.  
	By Lemma~\ref{compa}, we see that  
	\begin{equation} \label{plis} u^{(\gamma)}_{h,k+1}\leq u^{(\gamma)}_{h,k}\leq \gamma\, \Phi_\lambda^++C\,\Phi_\lambda^-\quad \mbox{in } \Omega_k\ \  \text{for every } k\geq k_0.\end{equation}
	As before, we obtain that, up to a subsequence, $u^{(\gamma)}_{h,k}$ converges to $u_h^{(\gamma)}$ in $C_{\text{loc}}^{1}(\Omega\setminus \{0\})$ as $k\to \infty$, where $u_h^{(\gamma)}$ is a non-negative solution of \eqref{boun1}. 
	
	From our choice of $C$, \eqref{nul} and Lemma~\ref{compa}, we infer that 
	$$ \underline{u}_\gamma(x)\leq C\,\Phi_\lambda^-(x)+ u_{h,k}^{(\gamma)}(x)\quad \mbox{for every }\  1/k\leq |x|\leq \delta.  
	$$ 
	By letting $k\to \infty$ and using that $\lim_{|x|\to 0} \underline{u}_\gamma(x)/\Phi_\lambda^+(x)=\gamma$, we arrive at $$ \liminf_{|x|\to 0} \frac{u_h^{(\gamma)}(x)}{\Phi_\lambda^+(x)}\geq \gamma .$$
	Moreover, from \eqref{plis}, we find that $\limsup_{|x|\to 0}u_h^{(\gamma)}(x)/\Phi_\lambda^+(x)\leq \gamma $. Hence,  
	$u_h^{(\gamma)}$ is a solution of \eqref{boun1}, subject to \eqref{var}. 
	
	\vspace{0.2cm}
	We next take $\gamma=0$ and assume that $h\not\equiv 0$ on $\partial\Omega$. Let $C>0$ be large so that the first inequality in \eqref{top} holds. As before, we consider
	\eqref{hati} (with $\gamma=0$) and obtain a non-negative solution $u_h^{(\gamma)}$ of \eqref{boun1} satisfying
	\eqref{plis}.    
	It follows that 
	$\lim_{|x|\to 0}u_h^{(\gamma)}(x)/\Phi_\lambda^+(x) =0$ and since $h\not\equiv 0$ on $\partial\Omega$, by the strong maximum principle, we infer that $u_h^{(\gamma)}>0$ in $\Omega$. By Theorem~\ref{th-cd}, we have $\lim_{|x|\to 0} |x|^{p_-}u_h^{(\gamma)}(x)\in (0,\infty)$. Moreover, there exists a unique solution for \eqref{boun1}, subject to \eqref{var} with $\gamma=0$. Indeed, if $u_0$ and $U_0$ are two such solutions, then by Lemma~\ref{compa}, we have $u_0\leq \varepsilon\,\Phi_\lambda^+  + U_0$ in $\Omega\setminus \{0\}$ for arbitrary $\varepsilon>0$. It follows that $u_0\leq U_0$ in $\Omega\setminus \{0\}$. By interchanging $u_0$ and $U_0$, we conclude that $u_0\equiv U_0$ in $\Omega\setminus \{0\}$.  
	
	\vspace{0.2cm}
	To finish the proof of Lemma~\ref{m2c}, it remains to 
	show that if $h\equiv 0$ on $\partial\Omega$ and $\gamma=0$, then \eqref{boun1}, subject to \eqref{var}, has no solutions. Indeed, if such a solution $u$ were to exist, then for every $\varepsilon>0$, we would have $u(x)\leq \varepsilon\, \Phi_\lambda^+(x)$ for every $x\in \Omega\setminus\{0\}$, which would lead to $u\equiv 0$ in $\Omega$ by letting $\varepsilon\to 0$. This is a contradiction. The proof of Lemma~\ref{m2c} is now complete. 
\end{proof}

\section{Proof of Theorem~\ref{mult1}} \label{multi}

As explained in Section~\ref{mainr}, by using the Kelvin transform, it suffices to establish the assertions of Theorem~\ref{mult1} in Case $(\mathcal M_2)$, which we state below.   

\begin{proposition}[Multiplicity, Case $(\mathcal M_2)$] \label{prol} Let $\Omega=\mathbb R^N$ and 
	Case~$(\mathcal M_2)$ hold, that is, $\lambda\leq \lambda_H$ and $\theta>\theta_+$. 
	Then, for every $\gamma\in (0,\infty)$,
	problem \eqref{eq1}, subject to 	$  \lim_{|x|\to 0} u(x)/\Phi_\lambda^{+}(x)=\gamma$, 
	has a unique solution $u_\gamma$,  
	which is radially symmetric and satisfies \eqref{kop2}, that is, $\lim_{|x|\to \infty} u_\gamma(x)/U_0(x)=1$. In addition, we have 
	\begin{equation} \label{aha} 
	\left\{ \begin{aligned}
	& u_\gamma\leq u_{\gamma'}\leq U_0\quad \mbox{ in } \mathbb R^N\setminus \{0\}\quad \mbox{for every } 0<\gamma<\gamma'<\infty,\\
	& U_0(x)=\lim_{\gamma\to \infty} u_{\gamma}(x)\quad \mbox{for each }x\in \mathbb R^N\setminus \{0\}. 
	\end{aligned}\right.
	\end{equation} 
	The set $\mathcal S$ of all positive solutions of \eqref{eq1} is 
	$ \mathcal S=\{U_0\}\cup \{u_\gamma:\ \gamma\in (0,\infty) \} 
	$. 	
\end{proposition}

\begin{proof} We split the proof into three steps. The first one deals with the existence and uniqueness of the solution $u_\gamma$ of \eqref{eq1}, subject to 	$  \lim_{|x|\to 0} u(x)/\Phi_\lambda^{+}(x)=\gamma$.  
	The claim of \eqref{aha} is proved in Step~2.  
	The asymptotic behavior of $u_\gamma$ in \eqref{kop2} follows from Theorem~\ref{mar2}, while the radial symmetry of $u_\gamma$ follows from uniqueness and radial symmetry of the problem \eqref{eq1}. In Step~3 we show that $U_0$ and $\{u_\gamma:\ \gamma\in (0,\infty)\}$ make up all the positive solutions of \eqref{eq1}.

	\vspace{0.2cm}
	\noindent {\bf Step 1.} For every $\gamma\in (0,\infty)$, there exists a unique solution $u_\gamma$ for problem \eqref{eq1}, subject to $  \lim_{|x|\to 0} u(x)/\Phi_\lambda^{+}(x)=\gamma$. 
	
	\vspace{0.2cm}
	\noindent {\em Proof of Step~1.}
	We fix $\gamma\in (0,\infty) $ and show 
	that any solutions 
	$u$ and $\widetilde{u}$ of \eqref{eq1}, subject to $  \lim_{|x|\to 0} u(x)/\Phi_\lambda^{+}(x)=\gamma$, must coincide. 
	Let $\varepsilon>0$ be arbitrary. 
	
	We define
	$ v_\varepsilon(x):=\left(1+\varepsilon\right) \widetilde{u}(x)+\varepsilon\, \Phi_\lambda^-(x)$ for every $x\in \mathbb R^N\setminus \{0\}$. 
	It is easy to check that $v_\varepsilon$ satisfies
	$$  -\mathbb L_\lambda( v_\varepsilon) + |x|^{\theta}\left(v_\varepsilon(x)\right)^q\geq 0\quad \mbox{for every } x\in \mathbb R^N\setminus \{0\}. 
	$$ 	
	From $\lim_{|x|\rightarrow 0} u(x)/\widetilde{u}(x)=1$, there exists $r_\varepsilon>0$ small such that 
	$$ u(x)\leq \left(1+\varepsilon\right) \widetilde{u}(x)\leq v_\varepsilon(x)\quad\mbox{for every  } 0<|x|\leq r_\varepsilon.$$ 
	The assumption $\theta>\theta_+$ yields that $\Theta>p_+\geq p_-$.  
	Hence, by Corollary~\ref{lem1}, there exists $R_\varepsilon>0$ large such that 
	$ u(x)\leq \varepsilon |x|^{-p_-}\leq v_\varepsilon(x)$ for every $|x|\geq R_\varepsilon$. Then, by Lemma~\ref{compa} with 
	$\omega:=\{x\in \mathbb R^N:\ r_\varepsilon<|x|<R_\varepsilon\}$, we find  
	that
	\begin{equation}\label{wee}
	u(x)\leq v_\varepsilon(x)\quad \text{for every }x\in \mathbb R^N\setminus \{0\}. 
	\end{equation}
	For $x\in \mathbb{R}^N\setminus \{0\}$ fixed, 
	letting $\varepsilon\to 0$ in \eqref{wee}, we obtain that $u\leq \widetilde u$ in $\mathbb{R}^N\setminus \{0\}$.
	By interchanging $u$ and $\widetilde{u}$, we conclude that $u\equiv \widetilde{u}$ in $\mathbb{R}^N\setminus \{0\}$.

	\vspace{0.2cm}
	We next prove that for arbitrary $\gamma\in (0,\infty)$, there exists a solution $u_\gamma$ of problem \eqref{eq1}, subject to $ \lim_{|x|\to 0} u(x)/\Phi_\lambda^{+}(x)=\gamma$. 
	By Theorem~\ref{est}, 
	for every $\gamma\in (0,\infty)$ and $k\geq 1$, there exists a unique solution $u_{k,\gamma}$ for the problem
	\begin{equation}\label{11}
	\left\{
	\begin{aligned}
	& -\mathbb L_\lambda(u)+|x|^{\theta}u^q=0&& \text{in } B_{k}(0)\setminus \{0\},&\\
	& \lim_{|x|\to 0}\frac{u(x)}{\Phi_\lambda^+(x)}=\gamma,\\
	&u=0&& \text{on }\partial B_k(0),&\\
	& u>0 && \text{in } B_k(0)\setminus\{0\}.&
	\end{aligned}
	\right.
	\end{equation} 
	Moreover, $u_{k,\gamma}$ is radially symmetric in $B_k(0)\setminus \{0\}$. 
	Recall that $U_0$ in \eqref{node} is a solution of \eqref{eq1} and $\lim_{|x|\to 0 }U_0(x)/\Phi_\lambda^+(x)=\infty$ since we are in Case $(\mathcal M_2)$.  
	By the comparison principle in Lemma~\ref{compa}, we have   
	\begin{equation} \label{fol}
	0<u_{k,\gamma}\leq u_{k+1,\gamma}\leq U_0\quad  \text{in } B_k(0)\setminus \{0\}. 
	\end{equation}
	By a standard argument, we deduce that, up to a subsequence, $u_{k,\gamma}\to u_\gamma$ in $C_{\text{loc}}^{1}(\mathbb{R}^N\setminus \{0\})$ as $k\to \infty$. Moreover, $u_\gamma$ is a radial solution of \eqref{eq1}. 
	From \eqref{fol} and  
	$ \lim_{|x|\to 0}u_{k,\gamma}(x)/\Phi_\lambda^+(x)=\gamma $ for each $k\geq 1$, 
	we find that 
	\begin{equation} \label{mib}
	\liminf_{|x|\to 0}\frac{u_\gamma(x)}{\Phi_\lambda^+(x)}\geq \gamma.\end{equation} 
	For every $\varepsilon>0$, we define $ w_\varepsilon(x):= 
	\left(\gamma+\varepsilon\right) \Phi_\lambda^+(x) + U_0(1)\, \Phi_\lambda^-(x) $ for every $0<|x|\leq 1$. Since $w_\varepsilon$ is a super-solution of \eqref{eq1} in $B_1(0)\setminus \{0\}$ such that $u_{k,\gamma}(x)\leq w_\varepsilon(x)$ whenever $|x|=1$, by Lemma~\ref{compa}, we deduce that 
	\begin{equation} \label{hak}
	u_{k,\gamma}(x)\leq w_\varepsilon(x)\quad \text{for every }0<|x|<1\ \ \mbox{and all } k\geq 1. 
	\end{equation}
	For $x\in B_1(0)\setminus\{0\}$ fixed, we 
	have $
	u_\gamma(x)\leq w_\varepsilon(x)$ by letting $k\rightarrow \infty$ in \eqref{hak}. This proves that $\limsup_{|x|\to 0} u_\gamma(x)/\Phi_\lambda^+(x)\leq \gamma+\varepsilon$. Letting $\varepsilon\to 0$, 
	jointly with \eqref{mib}, we find that 
	$ \lim_{|x|\to 0} u_\gamma(x)/\Phi_\lambda^{+}(x)=\gamma$.
	This ends the proof of Step~1. 
	
	
	

	\vspace{0.2cm} 
	\noindent {\bf Step 2.} Proof of \eqref{aha}.
	
	\vspace{0.2cm}
	\noindent {\em Proof of Step~2.} 
	Since $\lim_{|x|\to 0} U_0(x)/\Phi_\lambda^+(x)=\infty$ and $u_\gamma$ satisfies \eqref{kop2}, from Lemma~\ref{compa} we deduce that  
	\begin{equation} \label{unt} u_\gamma\leq u_{\gamma'}\leq U_0\quad \mbox{ in } \mathbb R^N\setminus \{0\}\quad \mbox{for every } 0<\gamma<\gamma'<\infty. 
	\end{equation}
	To show that $u_\gamma\to U_0$ pointwise in $\mathbb R^N\setminus \{0\}$ as $\gamma\to \infty$, it suffices to show that for every sequence $\{\gamma_j\}_{j\geq 1}$ with $\lim_{j\to \infty} \gamma_j=\infty$, there exists a subsequence of $\{u_{\gamma_j}\}$ (relabeled $\{u_{\gamma_j}\}$) that converges pointwise to $U_0$ in $ \mathbb R^N\setminus \{0\}$. Without loss of generality, we can assume that $\{\gamma_j\}_{j\geq 1}$ is increasing to $\infty$. Then, using \eqref{unt} as before, we find that, up to a subsequence, $\{u_{\gamma_j}\}_{j\geq 1}$ converges in $C^1_{\rm loc}(\mathbb R^N\setminus \{0\})$ to a positive solution $u_\infty$ of \eqref{eq1}, which satisfies $\lim_{|x|\to 0} u_\infty(x)/\Phi_\lambda^+(x)=\infty$. Then, Theorem~\ref{th-cd} gives that 
	$\lim_{|x|\to 0}u_\infty(x)/U_0(x)=1$. 
	With the same argument as in Step~1, we infer that $u_\infty\equiv U_0$ in $\mathbb R^N\setminus \{0\}$. This finishes Step~2. 
	
	\vspace{0.2cm} 
	\noindent {\bf Step 3.}  The set $\mathcal S$ of all solutions of \eqref{eq1} is 
	$ \mathcal S=\{U_0\}\cup \{u_\gamma:\ \gamma\in (0,\infty) \}. 
	$ 	 
	
	\vspace{0.2cm}
	\noindent {\em Proof of Step~3.}
	In Case $(\mathcal M_2)$, given any solution $u$ of \eqref{eq1}, exactly one of the alternatives {\bf (A)}, {\bf (B)} and {\bf (C)} in 
	Theorem~\ref{th-cd} holds. The alternative {\bf (C)} arises from the case
	$\lim_{|x|\to 0} u(x)/\Phi_\lambda^+(x)=\infty$. The fact that $U_0$ is the only positive solution of \eqref{eq1} in the situation {\bf (C)} proceeds exactly as in Step~1 for the uniqueness of $u_\gamma$. 
	Option {\bf (B)} corresponds to $u_\gamma$ with $\gamma\in (0,\infty)$. 
	
	We show that option ${\bf (A)}$ is not viable, that is, 
	problem \eqref{eq1} has no solutions satisfying
	\begin{equation} \label{sio} \lim_{|x|\to 0}\frac{u(x)}{\Phi_\lambda^+(x)}=0.\end{equation} 
	From the assumption $\theta>\theta_+$, we have $\Theta>p_+$. Hence, Theorem~\ref{mar2} gives that 
	\begin{equation} \label{mus}\lim_{|x|\to \infty} |x|^{p_+}\,u(x)=0.\end{equation} 
	
	We reach a contradiction by showing that $u=0$ in $\mathbb R^N\setminus\{0\}$. 
	Let $\varepsilon>0$ be arbitrary. 
	Given the definition of $\Phi_\lambda^+$ in \eqref{fundam}, we distinguish two cases: 
	
	\vspace{0.2cm}
	(i)
	Let 
	$\lambda<\lambda_H$. If $  V_\varepsilon(x):=\varepsilon\, \Phi_\lambda^+(x)=\varepsilon \,|x|^{-p_+}$ for each $x\in \mathbb R^N\setminus \{0\}$, then   
	$\lim_{|x|\to 0} u(x)/V_\varepsilon(x)=0$ in view of \eqref{sio}. 
	For every $x\in  \mathbb R^N\setminus \{0\}$, we have 
	\begin{equation} \label{mut} -\mathbb L_\lambda(V_\varepsilon)+|x|^\theta\,(V_\varepsilon)^q=|x|^\theta\,(V_\varepsilon)^q\geq 0. 
	\end{equation}
	Since $\lim_{|x|\to \infty} u(x)/V_\varepsilon(x)=0$, 
	by Lemma~\ref{compa}, we have 
	$$
	u(x)\leq V_\varepsilon(x)=\varepsilon\,|x|^{-p_+}\quad \mbox{for all } x\in \mathbb{R}^N\setminus \{0\}.$$ 
	By letting $\varepsilon\to 0$ we arrive at $u\equiv 0$ in $\mathbb R^N\setminus \{0\}$, which is a contradiction.
	
	\vspace{0.2cm}
	(ii) Let $\lambda=\lambda_H$, that is, $p_-=p_+=(N-2)/2$. 
	Then, from \eqref{mus}, for every $\varepsilon>0$ fixed, there exists 
	$R_\varepsilon>0$ large such that $u(x)\leq \varepsilon\,|x|^{-p_-}$ for every $|x|\geq R_\varepsilon$. 
	For every $0<|x|\leq R_\varepsilon$, we  define 
	$$ V_{\varepsilon}(x):= \frac{\varepsilon}{R_\varepsilon}|x|^{-p_-}\log \left(\frac{R_\varepsilon}{|x|}\right)+\varepsilon\, |x|^{-p_-}. 
	$$
	Using the definition of $\Phi_\lambda^+$ in \eqref{fundam}, 
	we remark that $V_\varepsilon$ satisfies \eqref{mut} for every $0<|x|< R_\varepsilon$. 
	From \eqref{sio}, we observe that $\lim_{|x|\to 0} u(x)/V_\varepsilon(x)=0$   
	and also $u(x)\leq \varepsilon |x|^{-p_-}=V_\varepsilon(x)$ for every $|x|=R_\varepsilon$. Thus, Lemma~\ref{compa} gives that 
	\begin{equation} \label{cool} u(x)\leq V_\varepsilon(x)\quad \mbox{for every } 0<|x|\leq R_\varepsilon. \end{equation}
	Now, $R_\varepsilon\to \infty$ as $\varepsilon\searrow 0$ so that for every fixed $x\in \mathbb R^N\setminus\{0\}$, we have $0<|x|<R_\varepsilon$ for every $\varepsilon>0$ small enough. Hence, by letting $\varepsilon\to 0$ in \eqref{cool}, we arrive at $u\equiv 0$ in $\mathbb R^N\setminus \{0\}$. 
	
	\vspace{0.2cm}
	Since in both cases we find a contradiction, 
	we conclude that in Case $(\mathcal M_2)$, equation \eqref{eq1} has no positive solutions with $\lim_{|x|\to 0} |x|^{p_-}\,u(x)\in (0,\infty)$. This finishes the proof of Step~3.

	\vspace{0.2cm}
	From Step~1--Step 3, we conclude the proof of Proposition~\ref{prol}. 
\end{proof}

\section{Proof of Theorem~\ref{non-e}} \label{sec5}

Let $\Omega=\mathbb R^N$. 
We show that \eqref{eq1} has no solutions in Case $(\mathcal N)$, that is, if $\lambda\leq \lambda_H$ and $ \theta_-\leq \theta\leq \theta_+$. 
Suppose by contradiction that $u$ is a solution of \eqref{eq1}. By Theorem~\ref{mar2}, we have
\begin{equation} \label{nun} 
\lim_{|x|\to\infty} |x|^{p_-}\,u(x)=0,
\end{equation}
whereas at zero, we derive from Theorem~\ref{th-cd} that 
\begin{equation} \label{nun2}
\begin{aligned}
& \lim_{|x|\to 0} |x|^{p_-} \,u(x)=0&& \mbox{ if } \theta=\theta_-,&\\
&  \lim_{|x|\to 0} |x|^{\Theta} \,u(x)=0 && \mbox{ if }\  \theta_-<\theta\leq \theta_+.&
\end{aligned}
\end{equation}
Let $\varepsilon>0$ be arbitrary. 
For all $x\in \mathbb R^N\setminus \{0\}$, we define 
$$ V_\varepsilon(x):=\left\{
\begin{aligned}
& \varepsilon |x|^{-p_-} && 
\mbox{ if } \theta=\theta_-,&\\
&  \varepsilon |x|^{-\Theta}+\varepsilon |x|^{-p_-}
&& \mbox{ if }\  \theta_-<\theta\leq \theta_+.& 
\end{aligned}
\right.$$ 
Recall that $\Phi_\lambda^-(x)=|x|^{-p_-}$ satisfies $\mathbb L_\lambda(\Phi_\lambda^-)=0$ in $\mathbb R^N\setminus \{0\}$. On the other hand, the assumption 
$\theta_-\leq \theta\leq \theta_+$ implies that $\ell\leq 0$, where $\ell$ is given by \eqref{sigma}, which means that $\varepsilon |x|^{-\Theta}$ is a super-solution of \eqref{eq1} since  
$$ -\mathbb L_\lambda(|x|^{-\Theta})=-\ell\, |x|^{-\Theta-2}\geq 0\quad \mbox{in }\mathbb R^N\setminus \{0\}. $$
Consequently, for every $\theta_-\leq \theta\leq \theta_+$, we have 
$$ -\mathbb L_\lambda(V_\varepsilon)+|x|^\theta\,(V_\varepsilon)^q\geq 0\quad \mbox{for every } x\in  \mathbb R^N\setminus \{0\}.
$$
From \eqref{nun} and \eqref{nun2}, we find that 
$$\lim_{|x|\to 0} u(x)/V_\varepsilon(x)=0\quad \text{and}\quad 
\lim_{|x|\to \infty} u(x)/V_\varepsilon(x)=0.$$ 
Hence, by the comparison principle in Lemma~\ref{compa}, we have 
\begin{equation}\label{ineqq}
u(x)\leq V_{\varepsilon}(x)\quad \text{for every }x\in \mathbb R^N\setminus \{0\}. 
\end{equation}
Fixing $x\in \mathbb{R}^N\setminus \{0\}$ and letting $\varepsilon\to 0$ in the above inequality, we see that $u\equiv 0$ in $\mathbb R^N\setminus \{0\}$, which is a 
contradiction with the positivity of $u$. This completes the proof of Theorem~\ref{non-e}.

\section{Comments and remarks} \label{com-rem}

In the study of equation \eqref{eq1}, it is customary to assume $\theta>-2$. Singular solutions arise precisely in this range, see Remark~\ref{obs}, and thus the methods available in the literature for analyzing problem \eqref{eq1} are often adapted to $\theta>-2$. In this paper, we treat \eqref{eq1} for every $\lambda,\theta\in \mathbb R$, by distinguishing four cases: $(\mathcal U)$, $(\mathcal M_1)$, $(\mathcal M_2)$ and $(\mathcal N)$. For the latter three cases, we compare $\theta$ with the critical exponents $\theta_\pm$ defined in \eqref{tat}. Hence, we need to carefully design our techniques based on sub/super-solutions so that they work in a unified manner, without depending on any behavior of the solutions near zero.  

To facilitate a comparison of our results with previous ones in the literature and to highlight the influence of the Hardy potential $\lambda\,|\cdot|^{-2}$ in \eqref{eq1}, we express the findings of this paper by  
treating $\theta\in \mathbb R$ and $q>1$ as fixed parameters and letting $\lambda$ vary in $\mathbb R$. From this perspective, when $\Omega=\mathbb R^N$ in \eqref{eq1} we gain a threshold for $\lambda$, denoted by $\lambda^*=\lambda^*(N,\theta,q)$ defined as follows
\begin{equation} \label{crit} 
\lambda^*:=\lambda_H-\left(\frac{N-2}{2}-
\frac{\theta+2}{q-1}\right)^2=\Theta\left(N-2-\Theta\right).\end{equation}
For problem \eqref{eq1} with $\Omega=\mathbb R^N$ (or for \eqref{boun}), a real number will be called a {\em threshold} for $\lambda$ if the existence of solutions  happens if and only if $\lambda$ is (strictly) greater than that number. 

Note that $\lambda^*\leq \lambda_{H}$ with equality if and only if $q=q_{N,\theta}$, where
we define 
\begin{equation} \label{qnt}  q_{N,\theta}:=\frac{N+2\,\theta+2}{N-2}.
\end{equation}
Clearly, $q_{N,\theta}>1$ if and only if  $\theta>-2$, which means that for $\theta\leq -2$, the structure of the solutions of \eqref{eq1} is less varied.  
Theorems~\ref{uniq}, \ref{mult1} and \ref{non-e} show that \eqref{eq1} with $\Omega=\mathbb R^N$ admits solutions if and only if $\lambda>\lambda^*$. If, moreover, $\lambda>\lambda_H$, then $U_0$ is the unique solution of \eqref{eq1} in $\mathbb R^N\setminus \{0\}$. On the other hand, when $\lambda^*<\lambda\leq \lambda_H$ (for $q\not=q_{N,\theta}$), then \eqref{eq1} with $\Omega=\mathbb R^N$ has infinitely many solutions and all its solutions are radially symmetric. Their asymptotic behavior near zero and at infinity is specified  in  Corollary~\ref{ou1}. 

\begin{corollary} \label{ou1} Fix $\theta\in \mathbb R$ and $q>1$. Let  $\lambda\in \mathbb R$ be arbitrary. 
	Problem \eqref{eq1} with $\Omega=\mathbb R^N$ has solutions if and only if $\lambda>\lambda^*$ and, in this case, the structure of all solutions is as follows:
	\begin{enumerate}
		\item[{\rm 1.}] If $\lambda>\lambda_H$, then $U_0$ in \eqref{node} is the only solution of problem \eqref{eq1}. 	 
		\item[{\rm 2.}] If $\lambda^*<\lambda\leq \lambda_H$ (whenever $q\not=q_{N,\theta}$), then 	
		all the solutions are radially symmetric and the set of all solutions of problem \eqref{eq1} is given by $$U_0\cup \{U_{\gamma,q,\lambda}:\ \gamma\in (0,\infty) \},$$ where we have 
		\begin{enumerate}
			\item[{\rm (a)}] If $q<q_{N,\theta}$ only for $\theta>-2$, then $U_{\gamma,q,\lambda}$ is the unique solution of problem \eqref{eq1} that satisfies 
			\begin{equation} \label{bio} \lim_{|x|\to 0} \frac{U_{\gamma,q,\lambda}(x)}{\Phi_\lambda^+(x)}=
			\gamma\in (0,\infty)\quad 
			\mbox{and} \quad  \lim_{|x|\to \infty} \frac{U_{\gamma,q,\lambda} (x)}{U_0(x)}=1.\end{equation} 
			\item[{\rm (b)}] If $q>\max\{q_{N,\theta},1\}$, then $U_{\gamma,q,\lambda}$ is the unique solution of problem \eqref{eq1} that satisfies
			\begin{equation} \label{m11}  \lim_{|x|\to 0} \frac{U_{\gamma,q,\lambda} (x)}{U_0(x)}=1 
			\ 
			\mbox{and} \  	
			\lim_{|x|\to \infty}   \frac{|x|^{N-2}\,U_{\gamma,q,\lambda}(x)}{\Phi_\lambda^+(1/|x|)}=\gamma\in (0,\infty).\end{equation} 
		\end{enumerate}     	
	\end{enumerate}		
\end{corollary}

\begin{remark}
	For $\lambda=0$, $q>1$ and $\theta>-2$, Corollary~\ref{ou1} shows the following: 
	
	(I) Problem \eqref{eq1} with $\Omega=\mathbb R^N$ has no solutions if $q\geq (N+\theta)/(N-2)$; this fact follows from \cite[Theorem~1.3]{CD2010} (with $p=2$ there) or from the celebrated paper of Brezis and V\'eron \cite{brv} when $\theta=0$, using also that every solution tends to zero at infinity by Corollary~\ref{lem1}.
	
	(II) Problem \eqref{eq1} with $\Omega=\mathbb R^N$ has infinitely many solutions, all radially symmetric, if $1<q<(N+\theta)/(N-2)$; moreover, in this case,  the set of all solutions is $U_0\cup \{U_{\gamma,q}:\ \gamma\in (0,\infty) \}$, where for each $\gamma\in (0,\infty)$, the (radially symmetric) solution $U_{\gamma,q}$ of \eqref{eq1} satisfies 
	$$ \lim_{|x|\to 0} |x|^{N-2}\,U_{\gamma,q}(x)=\gamma\quad \mbox{and}\quad 
	\lim_{|x|\to \infty}\, \frac{U_{\gamma,q}(x)}{U_0(x)}=1.$$ 
	Hence, for $\lambda=\theta=0$ in Case (II) above, we
	regain Theorem 3.2 of Friedman and V\'eron \cite{frv} (with $p=2$ there) and also reveal the precise rate at which $U_{\gamma,q}$ vanishes at infinity, namely 
	$$ \lim_{|x|\to \infty} |x|^{2/(q-1)}\,U_{\gamma,q}=\frac{2}{q-1}\left(\frac{2}{q-1}-N+2\right).$$ 
\end{remark} 

\begin{remark} \label{obs}
	(i) For $\theta=-2$ in Corollary~\ref{ou1}, we have $\lambda^*=0$ and $U_0\equiv \lambda^{1/(q-1)}$. 
	If $\theta=-2$ (respectively, $\theta<-2$), then whenever it exists, every solution of problem \eqref{eq1} with $\Omega=\mathbb R^N$ is radially symmetric and converges to $\lambda^{1/(q-1)}$ as $|x|\to 0$ (respectively, vanishes at zero precisely like $U_0$). 
	
	(ii) On the other hand, for $\theta>-2$ every solution of  problem \eqref{eq1} in $\mathbb R^N\setminus \{0\}$, whenever it exists (see Corollary~\ref{ou1}), is radially symmetric, blows-up at zero and vanishes at infinity. 
\end{remark}

We now review Theorem~\ref{est} on the
structure of all solutions for the problem 
\begin{equation} \label{bun} 
\left\{\begin{aligned}
& -\Delta u-\frac{\lambda}{|x|^2}u+|x|^{\theta}u^q=0\ \  \mbox{in }\Omega\setminus \{0\},&\\
& u=h\ \  \mbox{on } \partial\Omega,\quad \ u>0 \ \ \mbox{in } \Omega\setminus \{0\},&
\end{aligned}\right.
\end{equation} 
where $\Omega\subset \mathbb R^N$ is a smooth bounded domain containing the origin and $h$ is a non-negative and continuous function on $\partial\Omega$. We separate $\theta>-2$ from $\theta\leq -2$ to underscore the changes that occur when going from one case to the other. In Corollaries~\ref{ouk3} and \ref{ouk2} we assume $h=0$ in \eqref{bun}. 
When $h\not\equiv 0$ in \eqref{bun}, 
the structure of all solutions is discussed in Corollaries~\ref{oa2} and \ref{oa1}. 

\begin{corollary} \label{ouk3}
	Let $\theta> -2$, $q>1$ and $\lambda\in\mathbb R$. Define $q_{N,\theta}$ as in \eqref{qnt}. Suppose that $h=0$ in problem \eqref{bun}.  
	\begin{enumerate}
		\item[{\rm 1.}] If $\lambda>\lambda_H$, then problem \eqref{bun} has a unique solution $u_0$ and, moreover, we have
		$
		\lim_{|x|\to 0} u_0(x)/U_0(x)=1.$
		
		\item[{\rm 2.}] If $q\geq q_{N,\theta}$, then for every $\lambda\leq \lambda_H$, problem \eqref{bun} has no solutions.
		
		\item[{\rm 3.}] Let $1<q<q_{N,\theta}$.  
		\begin{enumerate}
			\item[{\rm (a)}] If $\lambda^*<\lambda\leq \lambda_H$, then
			all the solutions of problem  \eqref{bun} are given by
			$ \{u^{(\gamma)}:\ 0<\gamma\leq \infty \}$, where $u^{(\gamma)}$ is the unique solution of \eqref{bun}, subject to $\lim_{|x|\to 0} u(x)/\Phi_\lambda^+(x)=\gamma$. When $\gamma=\infty$, we have 
			$$\lim_{|x|\to 0} \frac{u^{(\gamma)}(x)}{U_0(x)}=1.$$
			\item[{\rm (b)}] If $\lambda\leq \lambda^*$, then problem
			\eqref{bun} has no solutions. 	
		\end{enumerate}
	\end{enumerate}
\end{corollary}

\begin{corollary} \label{ouk2}
	Fix $\theta\leq -2$ and $q>1$. Let $\lambda\in \mathbb R$ be arbitrary. Assume that $h=0$ in problem \eqref{bun}.  	
	\begin{enumerate}
		\item[{\rm 1.}] If $\lambda>\lambda_H$, then problem \eqref{bun} has a unique solution $u_0$ and, moreover, we have
		$\lim_{|x|\to 0} u_0(x)/U_0(x)=1$. 	
		\item[{\rm 2.}] If $\lambda\leq \lambda_H$, then problem \eqref{bun} has no solutions.   
	\end{enumerate}
\end{corollary}

\begin{remark} For problem \eqref{eq1} with $\Omega=\mathbb R^N$, we observe from Corollary~\ref{ou1} that $\lambda^*=\lambda^*(N,q,\theta)$ in \eqref{crit} is the threshold for $\lambda$ no matter how we fix $\theta\in \mathbb R$ and $q>1$. In addition, unless $q=q_{N,\theta}$ (relevant for $\theta>-2$), we see that the threshold $\lambda^*$ is {\em less} than $\lambda_H$.    
	
	On the other hand, when considering problem \eqref{bun} with $h=0$, the threshold for $\lambda$ becomes $\lambda_H$ 
	for every $\theta\in \mathbb R$ and $q> \max\{q_{N,\theta},1\}$. 
\end{remark}

To complete our comparison, we next consider $h\not\equiv 0$ in \eqref{bun}.

\begin{corollary} \label{oa2} Fix $\theta> -2$ and $q>1$. Assume $h\not\equiv 0$ in \eqref{bun}. Then, for every $\lambda\in \mathbb R$, problem \eqref{bun} has at least a solution $u_h$.
	
	\begin{enumerate} 
		\item[{\rm (I)}] There is only one solution $u_h$ exactly in the following cases:
		\begin{equation} \label{unim}
		{\rm (a)}\ \lambda>\lambda_H; \ \ {\rm (b)}\ \
		\lambda^*<\lambda\leq \lambda_H\ \ {\rm and}\ \  q>q_{N,\theta};\quad {\rm (c)} \ \lambda\leq \lambda^*.
		\end{equation}
		In cases {\rm (a)} and {\rm (b)}, 
		the solution $u_h$ satisfies $\lim_{|x|\to 0} u_h(x)/U_0(x)=1$.
		Furthermore, in case {\rm (c)}, we distinguish three situations:
		\begin{enumerate}
			\item[$(c_1)$] If $\lambda=\lambda^*$ and $q>q_{N,\theta}$, then $u_h$ satisfies \eqref{110}.  
			\item[$(c_2)$] If $\lambda=\lambda^*$ and $q=q_{N,\theta}$, then $u_h$ satisfies \eqref{111}.  
			\item[$(c_3)$] If ($\lambda=\lambda^*$ and $q<q_{N,\theta}$) or $\lambda<\lambda^*$, then 
			$\lim_{|x|\to 0} |x|^{p_-}\,u_h(x)\in (0,\infty)$. 
		\end{enumerate}
		\item[{\rm (II)}] For $\lambda^*<\lambda\leq \lambda_H$ and $q<q_{N,\theta}$, the set of 
		all solutions of \eqref{bun} is 
		$ \{u_h^{(\gamma)}:\ 0\leq \gamma\leq \infty \}$, where $u_h^{(\gamma)}$ is the unique solution of \eqref{bun}, subject to $\lim_{|x|\to 0} u(x)/\Phi_\lambda^+(x)=\gamma$. 
		We have $\lim_{|x|\to 0} u_h^{(\gamma)}(x)/U_0(x)=1$ if $\gamma=\infty$ and $\lim_{|x|\to 0} |x|^{p_-}\,u_h^{(\gamma)}(x)\in (0,\infty)$ if $\gamma=0$. 
	\end{enumerate}
\end{corollary}

\begin{corollary} \label{oa1}
	Fix $\theta\leq -2$ and $q>1$. Assume $h\not\equiv 0$ in \eqref{bun}. Then, for every $\lambda\in \mathbb R$, there exists a unique solution $u_h$ for \eqref{bun}. Moreover, we have:
	\begin{enumerate} 
		\item[{\rm (i)}]	If $\lambda>\lambda^*$, then $u_h(x)/U_0(x)\to 1$ as $|x|\to 0$. 
		\item[{\rm (ii)}] If $\lambda= \lambda^*$, then $u_h$ satisfies \eqref{110};
		\item[{\rm (iii)}] If $\lambda<\lambda^*$, then 	
		$\lim_{|x|\to 0} |x|^{p_-}\,u_h(x)\in (0,\infty)$.
	\end{enumerate}	
\end{corollary}

\begin{remark} \label{mis}
	Let $\theta,\lambda\in \mathbb R$ and $q>1$. No solution of \eqref{bun} with $h\not\equiv 0$ can be extended as a solution of \eqref{eq1} in $\mathbb R^N\setminus \{0\}$ in the situations below:
	\begin{enumerate}
		\item $\lambda\leq \lambda^*$;
		\item $\lambda>\lambda_H$ and $h\not\equiv U_0|_{\partial\Omega}$.
	\end{enumerate}
	On the other hand, if $\lambda^*<\lambda\leq \lambda_H$ (for $q\not=q_{N,\theta}$), then we see that 
	\begin{enumerate}
		\item for $q>\max\,\{q_{N,\theta},1\}$, the unique solution $u_h$ of \eqref{bun} can be extended to a solution of 
		\eqref{eq1} in $\mathbb R^N\setminus \{0\}$ provided that either $h=U_0|_{\partial\Omega}$ or $h=U_{\gamma,q,\lambda}|_{\partial\Omega}$ for some $\gamma\in (0,\infty)$, where $U_{\gamma,q,\lambda}$ is the unique solution of \eqref{eq1} in $\mathbb R^N\setminus \{0\}$ that satisfies \eqref{m11}.  
		\item for $q<q_{N,\theta}$ (only when $\theta>-2$) and every $\gamma\in (0,\infty]$ (but not for $\gamma=0$), the unique solution $u_h^{(\gamma)}$ of \eqref{bun}, subject to $\lim_{|x|\to 0} u(x)/\Phi_\lambda^+(x)=\gamma$ can be extended to a  
		solution of 
		\eqref{eq1} in $\mathbb R^N\setminus \{0\}$ provided that $h=U_0|_{\partial\Omega}$ for $\gamma=\infty$ and 
		$h=U_{\gamma,q,\lambda}|_{\partial\Omega}$ for $\gamma\in (0,\infty)$, where 
		$U_{\gamma,q,\lambda}$ is here the unique solution of \eqref{eq1} in $\mathbb R^N\setminus \{0\}$ that satisfies \eqref{bio}.  
	\end{enumerate}
	These facts follow by comparing Corollaries~\ref{oa2} and \ref{oa1} with Corollary~\ref{ou1}. 
\end{remark}

\end{document}